\let\accentvec\vec
\documentclass{svjour3}                     

\let\vec\accentvec
\smartqed  
\usepackage{graphicx}
\usepackage{fix-cm}
\usepackage[numbers]{natbib}
\usepackage[utf8]{inputenc}
\usepackage[T1]{fontenc}
\usepackage{lmodern}
\usepackage{latexsym}
\usepackage{multirow}
\usepackage{hyperref}
\usepackage{xcolor}
\usepackage{mathtools, amssymb}
\usepackage{subfigure}
\usepackage{phdalgo}
\usepackage{rotating}
\newtheorem{assumption}[theorem]{Assumption}
\newcommand{\setA}{\mathcal{A}} 
\newcommand{\setD}{\mathcal{D}} 
\newcommand{\cB}{\mathcal{B}}
\newcommand{\cF}{\mathcal{F}}

\newcommand{\funi}{u}

\definecolor{darkgreen}{rgb}{0.0 0.5 0.0}
\definecolor{whitegreen}{rgb}{0.0 0.75 0.0}
\definecolor{whiteblue}{rgb}{0.0 0.0 1.2}
\definecolor{darkred}{rgb}{0.8 0.0 0.0}
\definecolor{dkblue}{rgb}{0,0.1,0.5}
\definecolor{lightblue}{rgb}{0,0.5,0.5}
\definecolor{dkgreen}{rgb}{0,0.4,0}
\definecolor{dk2green}{rgb}{0.4,0,0}
\definecolor{dkviolet}{rgb}{0.6,0,0.8}
\definecolor{dkpink}{rgb}{0.8,0,0.9}

\newcommand{\N}{\mathbb{N}}					
\newcommand{\R}{\mathbb{R}}					
\newcommand{\Rbar}{\overline{\R}}	

\newcommand{\etc}{\textit{etc}~}
\newcommand{\ie}{\textit{i.e.}~}

\newcommand{\st}{s.t.~}

\newcommand{\resp}{resp.~}

\newcommand{\Snn}[1]{\mathcal{S}_{#1}} 
\newcommand{\Sn}{\Snn{n}}

\newcommand{\diag}{\text{diag }}

\renewcommand{\geq}{\geqslant}
\renewcommand{\leq}{\leqslant}

\renewcommand{\succeq}{\succcurlyeq}

\newcommand{\mons}[2]{\N_{#1}^{#2}}
\newcommand{\alphab}{\boldsymbol{\alpha}}

\newcommand{\npop}{n_\text{pop}}

\newcommand{\tlsc}{t^{\text{lsc}}}

\newcommand{\xopt}[7]{\xb_{#1} := (#2, #3, #4, #5, #6, #7)}
\newcommand{\fsa}{f_{\text{sa}}}
\newcommand{\fpop}{f_{\text{pop}}}

\newcommand{\Kpop}{K_{\text{pop}}}

\newcommand{\Xopt}{\mathbf{x}_\mathit{opt}}
\newcommand{\Xsdp}{\mathbf{x}_\mathit{sdp}}
\newcommand{\bop}{\mathtt{bop}}

\newcommand{\nbb}{\#{\text{boxes}}}

\newcommand{\Ksa}{K_{\mbox{\scriptsize sa}}}

\newcommand{\linf}{\text{sup}}
\newcommand{\lone}{L_{1}}

\numberwithin{equation}{section}

\DeclareMathOperator{\parab}{par}

\newcommand{\hessI}{\widetilde{\mathcal{D}^2(f)}}

\newcommand{\mI}[2]{\dfrac{#1+#2}{2}}
\newcommand{\dI}[2]{\dfrac{#1-#2}{2}}
\newcommand{\lmin}{\lambda_{\min}}
\newcommand{\lmax}{\lambda_{\max}}

\newcommand{\linfnorm}[1]{\Vert #1 \Vert_{\infty}}
\newcommand{\lonenorm}[1]{\Vert #1 \Vert_1}
\newcommand{\lonenormK}[1]{\|#1\|_{L_1(K)}}
\newcommand{\ltwonorm}[1]{\Vert #1 \Vert_2}
\newcommand{\comp}{\circ}

\newcommand{\ocaml}{\text{\sc OCaml}}

\newcommand{\intsolver}{\tt{intsolver}}

\newcommand{\sollya}{\mathtt{Sollya}}
\newcommand{\sparsepop}{\text{\sc SparsePOP}}

\newcommand{\precision}{p}
\newcommand{\dmax}{d_{\max}}

\newcommand{\itermax}{iter_{\max}}
\newcommand{\iter}{iter}

\newcommand{\updateprecision}{\mathtt{update\_precision}}
\newcommand{\guessargmin}{\mathtt{guess\_argmin}}
\newcommand{\argmin}{\mathtt{argmin}}
\newcommand{\randeval}{\mathtt{randeval}}

\newcommand{\astroot}{\mathtt{root}}

\newcommand{\composebop}{\mathtt{compose\_bop}}
\newcommand{\infnorm}{\mathtt{infnorm}}

\newcommand{\remez}{\mathtt{remez}}

\newcommand{\Df}{\mathcal{D}(f)}
\newcommand{\Hf}{\mathcal{D}^2(f)}
\newcommand{\Hphi}{\mathcal{D}^2(\phi)}
\newcommand{\dijlo}{\underline{d_{i j}}}
\newcommand{\dijup}{\overline{d_{i j}}}

\newcommand{\optim}{\mathtt{optim}}
\newcommand{\composeapprox}{\mathtt{compose\_approx}}
\newcommand{\saapprox}{\mathtt{reduce\_lift}}
\newcommand{\saapproxfun}[5]{\mathtt{reduce\_lift} (#1, #2, #3, #4, #5)}
\newcommand{\uopapprox}{\mathtt{unary\_approx}}

\newcommand{\dirset}{\mathit{D}}
\newcommand{\precset}{\mathcal{P}}

\newcommand{\tmalgofun}[3]{\mathtt{build\_quadratic\_form} (#1, #2, #3)}

\newcommand{\lift}{n_{\text{lifting}}}
\newcommand{\nt}{n_{\setD}}

\newcommand{\templateapproxfun}[4]{\mathtt{template\_approx} (#1, #2, #3, #4)}
\newcommand{\templateoptim}{\mathtt{template\_optim}}
\newcommand{\templateapprox}{\mathtt{template\_approx}}

\newcommand{\minsa}{\mathtt{min\_sa}}
\newcommand{\maxsa}{\mathtt{max\_sa}}

\newcommand{\xb}{\mathbf{x}} 

\newcommand{\yb}{\mathbf{y}} 
\newcommand{\zb}{\mathbf{z}}

\newcommand{\xc}{\mathbf{x}_c}

\newcommand{\Psa}{P^{\text{sa}}}

\newcommand{\omegatilde}[1]{\tilde{\omega}_{#1}}

\DeclareMathOperator{\radx}{rad}
\DeclareMathOperator{\sqrtop}{sqrt}
\newcommand{\radxtwo}{\radx_2}

\newcommand{\ext}[1]{\langle #1\rangle^{\text{sa}}} 
\newcommand{\dsa}{\ext{\setD}}

\begin{document}

\title{Certification of Real Inequalities
}
\subtitle{Templates and Sums of Squares}

\titlerunning{Certification of Real Inequalities -- Templates and Sums of Squares}        

\author{Xavier Allamigeon \and St\'ephane Gaubert \and Victor Magron \and Benjamin Werner
}

\institute{X. Allamigeon \at
              INRIA and CMAP \'Ecole Polytechnique \\
              Tel.: +33 (0)1 69 33 46 32\\
              \email{xavier.allamigeon@inria.fr}           
             \and
             S. Gaubert \at
              INRIA and CMAP \'Ecole Polytechnique \\
              Tel.: +33 (0)1 69 33 46 13\\
              \email{Stephane.Gaubert@inria.fr}
             \and
             V. Magron\at
              INRIA and LIX/CMAP \'Ecole Polytechnique \\
              Tel.: +33 (0)1 69 35 69 85\\
              \email{victor.magron@inria.fr}
              \and
              B. Werner \at
              LIX \'Ecole Polytechnique \\
              Tel.: +33 (0)1 69 33 41 41\\
              \email{benjamin.werner@polytechnique.edu}
}

\date{Received: date / Accepted: date}
\authorrunning{Xavier Allamigeon, St\'ephane Gaubert, Victor Magron and Benjamin Werner} 
\maketitle

\begin{abstract}
We consider the problem of certifying lower bounds for real-valued
multivariate transcendental functions. The functions we are dealing with are nonlinear and involve semialgebraic operations as well as some transcendental functions like $\cos$, $\arctan$, $\exp$, etc.
Our general framework is to use different approximation methods to relax the original problem into polynomial optimization problems, which we solve by sparse sums of squares relaxations. 
In particular, we combine the ideas of the maxplus approximations (originally introduced in optimal control) and of the linear templates (originally introduced
in static analysis by abstract interpretation). The nonlinear templates control the complexity of the semialgebraic relaxations at the price of coarsening the maxplus approximations. In that way, we arrive at a new - template based - certified global optimization method, which exploits both the precision of sums of squares relaxations and the scalability of abstraction methods. 
We analyze the performance of the method on 
problems from the global optimization literature, as well as medium-size inequalities issued from the Flyspeck project.

\end{abstract}

\section{Introduction}
\label{intro}

\subsection{Certification of Nonlinear Inequalities}
Numerous problems coming from different fields of mathematics (like combinatorics, geometry or group theory) have led to computer assisted proofs. One famous example is the proof of the Kepler conjecture, proved by Thomas Hales~\cite{Hales94aproof,halesannmath}. 
Recent efforts have been made to complete the formal verification of this conjecture. In particular, extensive computation are required to certify hundreds of nonlinear inequalities. We will often refer to the following inequality taken from Hales' proof:

\begin{example}[Lemma$_{9922699028}$ Flyspeck]
\label{ex:9922} Let $K$, $\Delta \xb$, $l$, $t$ and $f$ be defined as follows:
\[			
\begin{array}{rll}
K & :=&  [4, 6.3504]^3 \times [6.3504, 8] \times [4, 6.3504]^2 \enspace, \\
\Delta \xb  &:= & x_1 x_4 ( - x_1 +  x_2 +  x_3  - x_4 +  x_5 +  x_6) \\
&  & + x_2 x_5 (x_1  -  x_2 +  x_3 +  x_4  - x_5 +  x_6)  \\
&  & + x_3 x_6 (x_1 +  x_2  -  x_3 +  x_4 +  x_5  - x_6)  \\
&  & - x_2 x_3 x_4  -  x_1 x_3 x_5  -  x_1 x_2 x_6  - x_4 x_5 x_6 \enspace,\\
l(\xb) & := & -\pi / 2 +  1.6294  - 0.2213 (\sqrt{x_2}  +  \sqrt{x_3} +  \sqrt{x_5} +  \sqrt{x_6}  -  8.0) \\
& & + 0.913 (\sqrt{x_4}  -  2.52) + 0.728 (\sqrt{x_1}  -  2.0) \enspace,\\
t (\xb) & := & \arctan \frac{\partial_4 \Delta \xb }{\sqrt{4 x_1 \Delta \xb}} \enspace,\\
f (\xb) &:=&  l(\xb) + t (\xb) \enspace.
\end{array}
\]
Then, $\forall \xb \in K,  f(\xb)  \geq 0 \enspace.$
\end{example}
 
Note that the inequality of Example~\ref{ex:9922} would be
much simpler to check if $l$ was a constant (rather than a function of $\xb$). 
Indeed, semialgebraic optimization methods would provide precise lower
and upper bounds for the argument of $\arctan$. Then
we could conclude by monotonicity of $\arctan$ using interval
arithmetic. Here, both $l$ and $t$ depend on $\xb$. Hence, by
using interval arithmetic addition (without any domain subdivision) on
the sum $l + t$, which ignores the correlation between the argument of $\arctan$ and
the function $l$, we only obtain a coarse lower bound
(equal to $-0.87$, see Example~\ref{ex:9922_sa} for details);
 too coarse to assert the inequality . 
A standard way to improve this bound consists in subdividing the initial box (i.e. the Cartesian product of closed intervals) $K$ and performing interval arithmetic 
on smaller boxes. However, this approach suffers from the so called {\em curse of dimensionality}. Therefore, it is desirable to develop alternative certified global optimization methods, applicable to a wide class of problems involving semialgebraic and transcendental functions.

Moreover, the nonlinear inequalities of Flyspeck are challenging for numerical solvers for two reasons. First, they involve a medium-scale number of variables (6$\sim$10). Then, they are essentially {\em tight}. For instance, the function $f$ involved in Example~\ref{ex:9922} has a nonnegative infimum which is less than $10^{-3}$. The tightness of the inequalities to be certified is actually a frequent feature in mathematical proofs. Hence, we will pay a special attention
in the present work to {\em scalability} and {\em numerical precision} issues to provide certified bounds for global optimization. This is called {\em informal} certification as one obtains numerical certificates (e.g. sums-of-squares certificates) that can be in turn {\em formally} checked inside a proof assistant, such as Coq. 
The practical difficulties related to formalization (e.g. polynomial arithmetic implementation, appropriate certificate data-structures, formal proofs of approximations for nonlinear functions) are specifically addressed  in~\cite{jfr14}, as a further contribution of the authors.
\subsection{Nonlinear Global Optimization Problems}
Let $\dsa$ be the set of functions obtained by composing (multivariate) semialgebraic functions with special
functions taken from a {\em dictionary} $\setD$. We will typically include in $\setD$ the usual functions
$\tan$, $\arctan$, $\cos$, $\arccos$, $\sin$, $\arcsin$, $\exp$, $\log$,
$(\cdot)^{r}$ with $r\in \R\setminus\{0\}$. As we allow the composition with semialgebraic functions in our setting,
elementary functions like  $ +, -, \times, /,  |\cdot| ,\sup(\cdot,\cdot), \inf(\cdot,\cdot)$ are of course covered. 
Actually, we shall see that some of the present results remain valid if the dictionary includes semiconvex\footnote{Recall that for $\gamma \geq 0$, a function $\phi: \R^n\to \R$ is said to be {\em $\gamma$-semiconvex} if the function $\xb \mapsto \phi(\xb)+\frac{\gamma}{2}\ltwonorm{\xb}^2$ is convex.} or semiconcave functions with effective lower and upper bounds on the Hessian. More details about semiconcave functions can be found in~\cite{cannarsa2004semiconcave}.

Given $f, f_1, \dots, f_p \in \dsa$, we will address the following global optimization problem:
	
\begin{eqnarray}
\label{eq:fgeneral}
\inf_{\xb \in \R^n} & & f (\xb) \enspace,\\
\text{s.t.} & & f_1 (\xb) \geq 0, \dots, f_p (\xb) \geq 0 \enspace. \nonumber
\end{eqnarray}

The inequalities issued from Flyspeck actually deal with special cases
of computation of a certified lower bound for a real-valued multivariate
function  $f : \R^n \to \R$ over a compact semialgebraic set $K \subset \R^n$.
Checking these inequalities boils down to automatically provide lower bounds for the following instance of Problem~\eqref{eq:fgeneral}:
\begin{equation}
\label{eq:f}
f^*  :=  \inf_{\xb \in K} f (\xb) \enspace,
\end{equation}

We shall also search for {\em certificates} to assess that:
\begin{equation}
\label{eq:ineq}
\forall \xb \in K, f(\xb) \geq 0 \enspace.
\end{equation}

A well studied case is when $\setD$ is reduced to the identity map $\{Id\}$. 
Then, $f = \fsa$ belongs to the algebra $\setA$ of
semialgebraic functions (extension of multivariate polynomials with arbitrary compositions of $(\cdot)^{p}$, $(\cdot)^{\frac{1}{p}} (p \in \N_0)$, $\lvert\cdot\rvert$, $+$, $-$, $\times$, $/$, $\sup(\cdot,\cdot)$,
$\inf(\cdot,\cdot)$, where $\N_0$ stands for the set of positive integers) and Problem~\eqref{eq:fgeneral} specializes to the semialgebraic optimization problem:
\begin{equation}
\label{eq:fsa}
\fsa^*  :=  \inf_{\xb \in K} \fsa (\xb) \enspace.
\end{equation}

Another important sub-case is Polynomial Optimization Problems (POP), when $f = \fpop$ is a multivariate polynomial and $K=\Kpop$ is given by finitely many polynomial inequalities.
Thus, Problem \eqref{eq:fsa} becomes:
\begin{equation}
\label{eq:fpop}
\fpop^*  :=  \inf_{\xb \in \Kpop} \fpop (\xb) \enspace .
\end{equation}

We shall see that the presented methods also provide certified lower bounds (possibly coarse), for optimization problems which are hard to solve by traditional POP techniques. Such problems have a relatively large number of variables (10$\sim$100) or are polynomial inequalities of a moderate degree. For illustration purposes, we consider the following running example
coming from the global optimization literature.
\begin{example}[Modified Schwefel Problem 43 from Appendix B in~\cite{Ali:2005:NES:1071327.1071336}]
\label{ex:modified_swf} 
\[ \min_{\xb \in [1, 500]^n} f(\xb) = - \sum_{i = 1}^{n - 1} (x_i + \epsilon x_{i + 1}) \sin (\sqrt{x_i}) \enspace, \]
where $\epsilon$ is a fixed parameter in $\{0,1\}$. In the original problem, $\epsilon =
0$, \ie{}the objective function $f$ is the sum of independent
functions involving a single variable. This property may be exploited
by a global optimization solver by reducing it to the problem
$\min_{x \in [1, 500]} x \sin (\sqrt{x})$. Hence, we also consider a
modified version of this problem with $\epsilon = 1$.
\end{example}

\subsection{Certified Global Optimization in the Literature}

A common idea to handle
Problem~\eqref{eq:f} is to first approximate $f$ by multivariate
polynomials and then obtain a lower
bound of the resulting approximation by polynomial optimization techniques.

Computing lower bounds in constrained POP (see Problem\eqref{eq:fpop})
 is already a difficult problem, which has received much
attention. Sums of squares (SOS) relaxation based methods, leading to the resolution of semidefinite programs (SDP) have been developed
in~\cite{DBLP:journals/siamjo/Lasserre01, parrilo:polynomials}.
They can be applied to the more general class of semialgebraic
problems \cite{putinar1993positive}. 
Moreover, Kojima has developed a sparse refinement of the hierarchy of SOS relaxations (see~\cite{Waki06sumsof}). This has been implemented in the $\sparsepop$ solver.
Checking the validity of the lower bound of POP implies being able to control and certify the numerical error, as SDP solvers are typically implemented using floating point arithmetic. Such techniques rely on hybrid symbolic-numeric
certification methods, see Peyrl and Parrilo~\cite{DBLP:journals/tcs/PeyrlP08} and Kaltofen et al.~\cite{KLYZ09}. They allow one to produce positivity certificates for such POP. Alternative approaches to SOS/SDP are based
on Bernstein polynomials \cite{Zumkeller:2008:Thesis}.

The task is obviously more difficult in presence of transcendental functions.
Other methods of choice, not restricted to polynomial systems, include global optimization by interval methods (see e.g.~\cite{DBLP:journals/rc/Hansen06}), branch and bound methods with Taylor models~\cite{DBLP:journals/mp/CartisGT11, Berz:2009:RGS:1577190.1577198}.
Other methods involve rigorous Chebyshev approximations. An implementation of such approximations is available in the $\sollya$ tool~\cite{ChevillardJoldesLauter2010}. 

\subsection{Contribution}
\label{sec:contrib}

In this paper, we develop a general certification framework, combining methods from semialgebraic programming (SOS certificates, SDP relaxations) and from approximation theory. This includes classical methods like best uniform polynomials and less classical ones like maxplus approximation (inspired by optimal control and static analysis by abstract interpretation). 

The present approach exploits both the accuracy of SOS relaxations and the scalability of the approximation and abstraction procedure. This leads to a new method in global optimization, the nonlinear template method.
Namely, we alternate steps of semialgebraic approximation for some constituents of the objective function $f$ and semialgebraic optimization. The resulting constrained polynomial optimization problems are solved with sums of squares relaxation from Lasserre hierarchy, by calling a semidefinite solver. In this way, each iteration of the algorithms refines the following inequalities:
\begin{equation}
\label{eq:3relax}
f^* \geq f_{sa}^* \geq f_{pop}^* \enspace,
\end{equation}
where $f^*$ is the optimal value of the original problem,
$\fsa^*$ the optimal value of its current semialgebraic approximation
and $\fpop^*$ the optimal value of the SOS relaxation
which we solve. Under certain moderate assumptions, the lower bound $\fpop^*$ does
converge to $f^*$ (see Corollary~\ref{th:approx_cvg_minimizer}). 

The present nonlinear template method is an improved version of the maxplus approximation method originally presented in~\cite{victorecc}. By comparison, the new ingredient is the introduction of the template technique (approximating projections of the feasible sets), leading to
an increase in scalability.
This technique is an abstraction method, which is inspired by
the linear template of Sankaranarayanan, Sipma and Manna
in static analysis~\cite{Sankaranarayana+others/05/Scalable},
their nonlinear extensions by Adj\'e et al. \cite{adjegaubertgoubault11}.
As discussed below, it is closely related
to the maxplus basis methods, although the methods differ
in the way they propagate approximations. 

In the present application, templates are used
both to approximate transcendental functions,
and to produce coarser but still tractable relaxations
when the standard SOS relaxation of the semialgebraic
problem is too complex to be handled. As a matter of fact,
SOS relaxations are a powerful tool to get tight certified lower bound
for semialgebraic optimization problems, but applying them is currently
limited to small or medium size problems: their execution
time grows exponentially with the relaxation order, which itself
grows with the degree of the polynomials involved in the semialgebraic relaxations. The template method
allows to reduce these degrees, by approximating certain projections
of the feasible set by a moderate number of nonlinear inequalities.


In this article, we present the following approximation schemes:
\begin{itemize}
\item \textbf{Semialgebraic maxplus templates for multivariate transcendental functions} This method uses maxplus approximation of semiconvex transcendental functions by quadratic functions. 
The idea of maxplus approximation comes from optimal control: it was originally introduced
by Fleming and McEneaney~\cite{a5} and developed by several authors \cite{a6,a7,curseofdim,PhysRevA.82.042319,conf/cdc/GaubertMQ11}, to represent the value function by a ``maxplus linear combination'', which
is a supremum of certain basis functions, like quadratic polynomials.  When applied to the present context, this idea leads to approximate from above and from below every transcendental function appearing in the description of the problem by infima and suprema of finitely many quadratic polynomials. In that way, we are reduced to a converging sequence
of semialgebraic problems. A geometrical way
to interpret the method is to think of it
in terms of ``quadratic cuts'' quadratic inequalities
are successively added to approximate the graph
of a transcendental function (Sect.~\ref{sec:template_approx}). 
\item \textbf{Non-convex quadratic templates} Sub-components of the objective function $f$ (\resp{}its semialgebraic approximations) are replaced by suprema of quadratic polynomials  (Sect.~\ref{sec:quadr_templates}).
\item \textbf{Polynomial under-approximations for semialgebraic functions} Given a degree $d$ and a semialgebraic sub-component $\fsa$ of $f$ that involves a large number of lifting variables, we build a hierarchy of polynomial approximations, that converge to the best (for the $\lone$ norm) degree-$d$ polynomial under-approximation of $\fsa$ (Sect.~\ref{sec:l1_under}).
\end{itemize}
The paper is organized as follows. In Sect.~\ref{sec:sos},
we recall the definition and properties of Lasserre relaxations
of polynomial problems, together with reformulations by Lasserre
and Putinar of semialgebraic problems classes. The  maxplus approximation and the nonlinear templates are presented in Sect.~\ref{sec:templates}.
In Sect.~\ref{sec:template_algo}, we describe the nonlinear template optimization algorithm together with the convergence study of the method.
The main numerical results are presented in Sect.~\ref{sec:benchs}. 

\section{Application of SOS to Semialgebraic Optimization}
\label{sec:sos}
Let $\R_d[\xb]$ be the vector space of real forms in $n$ variables of degree $d$ and $\R[\xb]$ the set of multivariate
polynomials in $n$ variables. We also define the cone $\Sigma_{d} [\xb]$ of sums of squares of degree at most $2 d$.
\subsection{Constrained Polynomial Optimization Problems and SOS}
\label{sec:pop}
We consider the general constrained polynomial optimization problem (POP):
\begin{equation}
\label{eq:cons_pop}
\fpop^*  :=  \inf_{\xb \in \Kpop} \fpop (\xb) \enspace,
\end{equation}
where $\fpop : \R^n \to \R$ is a $d$-degree multivariate polynomial,
$\Kpop$ is a compact set defined by polynomials inequalities $g_1(\xb)
\geq 0,\dots,g_m(\xb) \geq 0$ with $g_j(\xb) : \R^n \to \R$ being a
real-valued polynomial of degree $w_j, j = 1,\dots,m$.  We call
$\Kpop$ the feasible set of Problem~\eqref{eq:cons_pop}.
Let $g_0 := 1$. We introduce the $k$-truncated quadratic module $QM_k(\Kpop) \subset \R_{2 k}[\xb]$ associated with $g_1, \cdots, g_m$:
\begin{align*}
QM_k (\Kpop) = \Bigl\{\sum_{j=0}^m \sigma_j(\xb) g_j (\xb): \sigma_j \in \Sigma_{k - \lceil w_j / 2\rceil}[\xb]  \Bigr\} \enspace ,
\end{align*} 
and define the quadratic module $QM(\Kpop) := \bigcup_{k \in \N} QM_k(\Kpop)$.

\begin{definition}
A quadratic module $M$ is called {\em archimedean} if $N - \|\xb\|_2^2 \in M$ for some $N \in \N_0$.
\end{definition}

Let $k \geq k_0 := \max( \lceil d / 2 \rceil, \max_{0 \leq j \leq m} \{\lceil w_j / 2\rceil\} )$ and consider the hierarchy of semidefinite relaxations $Q_k: \sup \{ \mu : \fpop(x) - \mu  \in QM_k(\Kpop)\,, \ \mu \in \R  \}$,
\if{
\[
Q_k:\left\{			
\begin{array}{ll}
\sup\limits_{\mu \in \R} & \mu \\			 
\text{s.t.} & \fpop(x) - \mu  \in QM_k(\Kpop) \enspace,\\
\end{array} \right.
\]
}\fi
with optimal value denoted by $\sup (Q_k)$. The integer $k$ refers to the SOS relaxation order. 
\begin{theorem}
\label{prop:sos_cvg}
The sequence of optimal values $(\sup (Q_k))_{k \geq k_0}$ is
non-decreasing. If the quadratic module $QM(\Kpop)$ is archimedean, then this sequence converges to $\fpop^*$. 
\end{theorem}	
\begin{proof}
The proof follows from~\cite[Theorem 4.2 (a)]{DBLP:journals/siamjo/Lasserre01} as~\cite[Assumption 4.1]{DBLP:journals/siamjo/Lasserre01} is satisfied when $QM(\Kpop)$ is archimedean.\qed
\end{proof}
The non-linear inequalities to be proved in the Flyspeck project typically
involve a variable $\xb$ lying in a box. Thus, Theorem~\ref{prop:sos_cvg} applies when one adds to the definition of $\Kpop$ the redundant constraint $g (\xb) := N - \|\xb\|_2^2 \geq 0$ for some large enough $N \in \N_0$.

\subsection{Semialgebraic Optimization}
\label{sec:sa}
In this section, we recall how the previous approach can be extended to semialgebraic optimization problems by introducing lifting variables.
The set $\setA$ of semialgebraic functions $\fsa : \Ksa \mapsto \R$ is the algebra generated by finite composition of the operations $+, -, \times, /, \sup, \inf, |\cdot|, (\cdot)^{\frac{1}{p}} (p \in \N_0)$ on polynomials, whenever these operations are well-defined (e.g. division by zero never occurs).
Let consider the problem 
\begin{equation}
\label{eq:cons_fsa}
\fsa^* = \inf_{\xb \in \Ksa} \fsa (\xb) \enspace, 
\end{equation}
where $\Ksa := \{\xb \in \R^n \, : \, g_1(\xb) \geq 0, \dots, g_m(\xb) \geq 0 \}$ is a basic semialgebraic set. 

\begin{definition}[Basic Semialgebraic Lifting]
A semialgebraic function $\fsa$ is said to have a basic semialgebraic lifting if there exist $p, s \in \N$, polynomials $ h_1, \dots , h_s \in \R[\xb, z_1,\dots,z_p]$ and a basic semialgebraic set $\Kpop$ defined by:
\[\Kpop :=  \{ (\xb,  z_1,\dots,z_p) \in \R^{n+p} : \xb \in \Ksa, h_1(\xb, \zb) \geq 0,\dots, h_s(\xb, \zb) \geq 0 \} \enspace, \]
such that the graph of $\fsa$ (denoted $\Psi_{\fsa}$) satisfies:
\[\Psi_{\fsa} := \{ (\xb, \fsa(\xb)) : \xb \in \Ksa\} = \{ (\xb, z_p) : (\xb,  \zb) \in \Kpop\} \enspace. \]
\end{definition}
By~\cite[Lemma 3]{DBLP:journals/siamjo/LasserreP10}, every function $\fsa \in \setA$\footnote{We presume that in~\cite[Lemma 3]{DBLP:journals/siamjo/LasserreP10}, ``well-defined function $f$'' stands for the fact that $f$ can be evaluated in a non-ambiguous way on the considered domain.} has a basic semialgebraic lifting. 
\if{
\begin{lemma}[Lasserre, Putinar~\cite{DBLP:journals/siamjo/LasserreP10}]
\label{lemma:bsal}
Every well-defined $\fsa \in \setA$ has a basic semialgebraic lifting.
\end{lemma}
}\fi
To ensure that the Archimedean condition is preserved, we add bound constraints over the lifting variables. These bounds are computed by solving semialgebraic optimization sub-problems. All the semialgebraic functions involved in Flyspeck inequalities have a basic semialgebraic lifting.

\begin{example} [from Lemma$_{9922699028}$ Flyspeck]
\label{ex:9922_sa}
Continuing Example~\ref{ex:9922}, we consider the function
$\fsa := \frac{\partial_4 \Delta \xb }{\sqrt{4 x_1 \Delta \xb}}$ and the set $\Ksa := [4, 6.3504]^3 \times [6.3504, 8] \times [4, 6.3504]^2$. The latter can be equivalently rewritten as 
\[
\Ksa := \{ \xb \in \R^6 : g_1 (\xb) \geq 0, \dots, g_{12} (\xb) \geq 0\} \enspace,
\]
where $g_1(\xb) := x_1 - 4, g_2(\xb) := 6.3504 - x_1, \dots , g_{11}(\xb) := x_6 - 4, g_{12}(\xb) := 6.3504 - x_6$.

We introduce two lifting variables $z_1$ and $z_2$, respectively representing the terms $\sqrt{4 x_1 \Delta \xb}$ and $\frac{\partial_4 \Delta \xb }{\sqrt{4 x_1 \Delta \xb}}$. 

We also use a lower bound $m_1$ of $ \inf_{\xb \in \Ksa} \sqrt{4 x_1 \Delta \xb}$ and an upper bound $M_1$ of $ \sup_{\xb \in \Ksa} \sqrt{4 x_1 \Delta \xb}$ which can be both computed by solving auxiliary subproblems. 

Now the basic semialgebraic set $\Kpop$ and the graph $\Psi_{\fsa}$ of $\fsa$ can be defined as follows:
\begin{align*}
\Kpop := {} & \{ (\xb, z_1 ,z_2) \in \R^{6+2} : 
\xb \in \Ksa, \, h_j(\xb, z_1, z_2) \geq 0,
j = 1,\dots,6 \} \enspace,  
\\
\Psi_{\fsa} := {} & \{ (\xb, \fsa(\xb)) : \xb \in \Ksa\} =  \{ (\xb, z_2) : (\xb, z_1, z_2) \in \Kpop\} \enspace,
\end{align*}
where the multivariate polynomials $h_j$ are defined by:
\begin{align*}
h_1(\xb,  \zb) & := z_1 - m_1 \enspace, & h_4(\xb,  \zb) & := - z_1^2 + 4 x_1 \Delta \xb \enspace, \\
h_2(\xb,  \zb) & := M_1 - z_1 \enspace, & h_5(\xb,  \zb) & := z_2 z_1 - \partial_4 \Delta \xb \enspace, \\
h_3(\xb,  \zb) & := z_1^2 - 4 x_1 \Delta \xb \enspace, & h_6(\xb,  \zb) & := - z_2 z_1 + \partial_4 \Delta \xb\enspace.
\end{align*}
Let $h_0 := 1, \omega_l := \deg h_l \ (0 \leq l \leq 6)$. 
Consider the following semidefinite relaxations:
\[
Q_k^{sa}: \left\{			
\begin{array}{ll}
\max\limits_{\mu, \sigma_j, \theta_l} & \mu \\			
\text{s.t.} & z_2 - \mu = \sum_{j=1}^{12} \sigma_j(\xb, \zb) g_j (\xb) + \sum_{l=0}^6 \theta_l(\xb, \zb) h_l (\xb, \zb), \enspace \forall (\xb, \zb) \enspace, \\
            & \sigma_j \in \Sigma_{k - 1}[\xb, \zb], \enspace 1 \leq j \leq 12 \enspace, \\
            & \theta_l \in \Sigma_{k - \lceil \omega_l / 2\rceil}[\xb, \zb], 0 \leq l \leq 7 \enspace .
\end{array} \right.
\]
When $k \geq k_0 := \max_{1 \leq j \leq 6} \{\lceil\omega_j / 2\rceil \}= 2$, then as a special case of Theorem~\ref{prop:sos_cvg}, the sequence $(\inf (Q_k^{sa}))_{k \geq 2}$ is monotonically non-decreasing and converges to $\fsa^*$. The lower bound $m_2 = -0.618$ computed  
at the $Q_2^{sa}$ relaxation is too coarse. A tighter lower bound $m_3 = -0.445$ is obtained at the third relaxation, but it consumes more CPU time.
\end{example}

\section{Maxplus Approximations and Nonlinear Templates}
\label{sec:templates}

\subsection{The Basis of Maxplus Functions} 
\label{sec:maxplusbasis}
Let $\cB$ be a set of functions $\R^n\to \R$, whose elements
will be called \textit{maxplus basis functions}. Given
a function $f:\R^n\to \R$, we look
for a representation of $f$ as a linear combination
of basis functions in the maxplus sense, i.e.,
\begin{equation}
\label{eq:fenchel}
f = \sup_{w\in \cB}(a(w) + w) \enspace, 
\end{equation}
where $(a(w))_{w\in \cB}$ is a family of elements of $\R \cup\{-\infty\}$
(the ``coefficients''). 
The correspondence between the function $x \mapsto f(x)$ and the 
coefficient function $w\mapsto a(w)$ is a well studied
problem, which has appeared in various
guises (Moreau conjugacies, generalized Fenchel transforms, Galois correspondences, see~\cite{agk04} for more background).

The idea of maxplus approximation~\cite{a5,mceneaney-livre,a6} is to
choose a space of functions $f$ and a corresponding set $\cB$ of
basis functions $w$ and to approximate from below a given $f$ in this space by a finite maxplus linear combination,
\(
f \simeq  \sup_{w\in \cF}(a(w) + w) \enspace, 
\)
where $\cF\subset\cB$ is a finite subset. Note that 
$ \sup_{w\in \cF}(a(w) + w) $ is not only an approximation but a valid
lower bound of $f$. This is reminiscent of classical linear approximation
methods and in particular of the finite element methods, in which
a function in an finite dimensional space is approximated by a linear combination of prescribed elementary functions. Note that the term ``basis''
is abusive in the maxplus setting, as the family of functions $w\in \cF$
is generally not free in the tropical sense. 

A convenient choice of maxplus basis functions is the following~\cite{a5,a6}.
For each constant
$\gamma\in \R$, we shall consider 
the family of quadratic functions $\cB=\{w_{\yb} \mid \yb \in \R^n\}$, where 
\begin{align}
w_{\yb}(\xb):= -\frac{\gamma}{2} \ltwonorm{\xb-\yb}^2  \enspace .
\label{e-quadraticbasis}
\end{align}
Whereas in classical approximation problems, the ambient function spaces of interest are Sobolev spaces $H^k$,
or spaces $\mathcal{C}^k$ of $k$ times differentiable functions, in the tropical
settings, the appropriate spaces, consistent with the choice of quadratic
maxplus basis functions, turn out to consist of {\em semiconvex functions},
which we next examine.

\subsection{Maxplus Approximation for Semiconvex Functions}
\label{sec:samp_lsc}

The following definition is standard in variational analysis.
\begin{definition}[Semiconvex function]
Let $\gamma$ denote a nonnegative constant.
A function $\phi: \R^n\to \R$ is said to be {\em $\gamma$-semiconvex} if the function $\xb \mapsto \phi(\xb)+\frac{\gamma}{2}\ltwonorm{\xb}^2$ is convex. 
\end{definition}
\begin{proposition}
Let $\cB$ denote the set of quadratic functions $w_{\yb}$
of the form \eqref{e-quadraticbasis} with $\yb \in \R^n$.
Then, the set of functions $f$ which can be written as a maxplus linear
combination~\eqref{eq:fenchel} for some function $a: \cB \to \R\cup\{-\infty\}$ is precisely the set of lower semicontinuous $\gamma$-semiconvex functions. 
\end{proposition}
\begin{proof}
Let us note $h^\star:  \R^n \to \R\cup \{\pm \infty\}$ the Legendre-Fenchel\index{Legendre-Fenchel transform} transform of a function $h: \R^n \to \R\cup \{\pm \infty\}$, so that 
$h^\star(p):=\sup_{x\in \R^n} \langle p,x \rangle -h(x).$ 
A known fact is that a convex lower semicontinuous function $g: \R^n
\to \R\cup\{\pm\infty\}$ 
is the supremum of the affine functions that it dominates~\cite[Th.~12.1]{rockafellar}. Actually, it is shown there that 
$g(x) = g^{\star\star}(x)= \sup_{p\in \R^n} \langle p, x\rangle -g^\star(p)$.
By applying this result to the function  $g(x) = f(x) + \frac{\gamma}{2}\|x\|_2^2$, we deduce that 
$f(x) = \sup_{p\in \R^n} \langle p, x\rangle  - \frac{\gamma}{2}\|x\|_2^2 -g^\star(p)
= \sup_{p\in \R^n} -\frac{\gamma}{2}\|x-\frac{1}{\gamma}p\|_2^2 -g^\star(p)+ \frac{1}{2\gamma}\|p\|_2^2$.
\if{
\begin{align*}
f(x) &= \sup_{p\in \R^n} \langle p, x\rangle  - \frac{\gamma}{2}\|x\|_2^2 -g^\star(p)\\
&= \sup_{p\in \R^n} -\frac{\gamma}{2}\|x-\frac{1}{\gamma}p\|_2^2 -g^\star(p)+ \frac{1}{2\gamma}\|p\|_2^2 \enspace,
\end{align*}
}\fi
which is of the form~\eqref{e-quadraticbasis}. 

Conversely, 
since an arbitrary supremum
of $\gamma$-semiconvex and lower semicontinuous \index{lower semicontinuous}
is also $\gamma$-semiconvex and lower semicontinuous, the
supremum in~\eqref{e-quadraticbasis} defines a 
$\gamma$-semiconvex and lower semicontinuous function.\qed
\end{proof}

The transcendental functions which we consider here are twice
continuously differentiable. Hence, their restriction to any bounded convex set 
is $\gamma$-semiconvex for a sufficiently large $\gamma$, so that they can be approximated by finite suprema
of the form $\sup_{w\in \cF}(a(w) + w) $ with $\cF\subset \cB$.

The following result is derived
in~\citep[Theorem 3.2]{conf/cdc/GaubertMQ11} using methods and results
of Gr\"uber~\cite{b2}, who studied the best approximation of a convex
body by a polytope. 
It shows that if $N=|\cF|$
basis functions are used, then the best approximation
error is precisely of order $1/N^{2/n}$ (the error is the sup-norm, over any compact set),
provided that the function to be approximated is of class $\mathcal{C}^2$.
We call $\Hphi(\xb)$ the Hessian matrix of $\phi$ at $\xb$ and 
suppose that we approximate the function $\phi$ by the finite supremum
of $N$ $\gamma$-semiconvex functions
parametrized by $p_i (i = 1, \dots, N)$ and $a_i (i = 1, \dots, N)$:
\[ \phi \simeq \tilde{\phi}_N := \max\limits_{1 \leq i \leq N } \{ \frac{\gamma}{2}\ltwonorm{\xb}^2 + p_i^T \xb + a(p_i)  \}\enspace. \enspace\]

\begin{theorem}[$\linf$ approximation error, {\citep[Theorem 3.2]{conf/cdc/GaubertMQ11}}]
\label{th:maxpluslinf}
Let $\gamma \in \R$, $\epsilon >0$ and let $K \subset \R^n$ denote any full dimensional compact convex subset. If $\phi: \R^n \mapsto \R$  is $(\gamma - \epsilon)$-semiconvex of class $\mathcal{C}^2$, then there exists a positive constant $\alpha$ depending only on $n$ such that:
\[ \linfnorm{\phi - \tilde{\phi}_N} \sim \dfrac{\alpha}{N^{2/n}} \Bigl(\int_K [\det (\Hphi(\xb) + \gamma I_n)]^{\frac{1}{2}} d\xb  \Bigr)^{\frac{2}{n}} \text{ as } N \rightarrow \infty \enspace .\]
\end{theorem}

Thus, the best approximation satisfies
\begin{align}
\linfnorm{\phi - \tilde{\phi}_N} \simeq \frac{C(\phi)}{N^{2/n}} \enspace,
\label{e-pessimistic}
\end{align}
where the constant $C(\phi)$ is explicit (it depends of $\det (\Hphi+ \gamma I_n)$
and is bounded away from $0$ when $\epsilon$ is fixed). 
This approximation
indicates that some curse of dimensionality is unavoidable:
to get a uniform error of order $\epsilon$, one needs
a number of basis functions of order $1/\epsilon^{n/2}$. 
Equivalently, the approximation error is of order $O(h^{\frac{2}{n}})$ where $h$ is a space discretization step.
The assumption that $\tilde{\phi}_N$ is of class $\mathcal{C}^2$ in Theorem~\ref{th:maxpluslinf} is needed to obtain a tight asymptotics of the approximation error. However, the max-plus approximation error is known to be of order
$O(N^{2/n})$ under milder assumptions,
requiring only semi-convexity type condition, see Proposition~64
of~\cite{lakhouathesis}, and also Lemma~16 of~\cite{a6} for a coarser approximation in $O(N^{1/n})$ valid in more general circumstances.
This is due to the asymmetrical character of the maxplus
approximation (a ``one-sided'' regularity, captured by the
semiconvexity condition, is involved). Thus, unlike Taylor
models, max-plus approximation does not require a
$\mathcal{C}^k$ type regularity. For instance, a nonsmooth
function like $|x|-x^2/2=\max(x-x^2/2,-x-x^2/2)$
can be perfectly represented by two quadratic max-plus basis functions.
In what follows, we shall always apply the approximation 
to small dimensional constituents of the optimization
problems. 

In this way, starting from a transcendental univariate
elementary function $f \in \setD$,  such as $\arctan$, $\exp$, \etc{},
defined on a real bounded interval $I$, we arrive at a 
semialgebraic lower bound of $f$, which is nothing but
a supremum of a finite number of quadratic functions.

\begin{example}
Consider the function $f= \arctan$ on an interval $I := [m, M]$.
For every point $a\in I$, we can find a constant $\gamma$
such that
\[
\arctan (x)  \geq \parab_{a}^-(x):=  -\frac{\gamma}{2} (x-a)^2 +f'(a) (x - a) + f (a)
\enspace .
\]
Choosing $\gamma=\sup_{x\in I} -f''(x)$ always work. However, 
it will be convenient to allow $\gamma$ to depend on the choice of $a$ to get tighter lower bounds. 
Choosing a finite subset $A\subset I$, we arrive at an approximation
\begin{equation}
\label{eq:max_par}
\forall x \in I , \, \arctan \, (x) \geq \max_{a\in A} \, \parab_{a}^-(x) \enspace .
\end{equation} 
Semialgebraic over-approximations $x \mapsto \min_{a \in A} \parab_{a}^+(x)$ can be defined in a similar way. 
Examples of such under-approximations and over-approximations are depicted in Fig.~\ref{fig:samp_atn}.
\begin{figure}[!ht]	
\begin{center}
\includegraphics[scale=0.7]{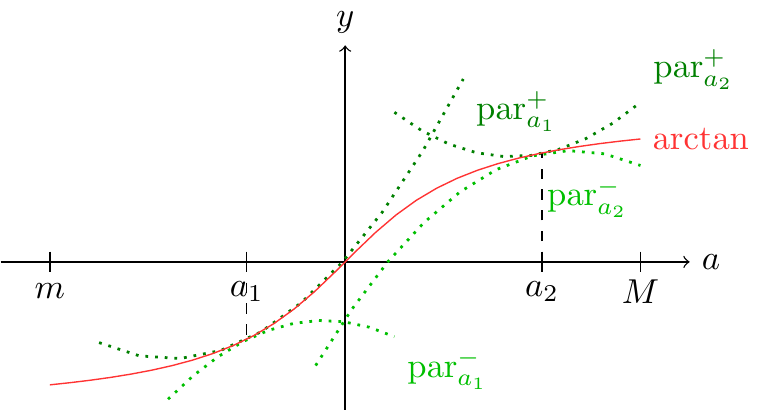}
\caption{Semialgebraic under-approximations and over-approximations for $\arctan$}	\label{fig:samp_atn}		
\end{center}
\end{figure}
\end{example}
\vspace*{-0.5cm}
\begin{example}
Consider the bivariate function $g : (x_1, x_2) \mapsto \sin(x_1 + x_2)$, defined on $K := [-1.5,4]\times[-3,3]$, which is a component of the objective function from Problem MC (see Appendix~\ref{sec:appa}). As in the previous example, we can build under-approximations for the sine function. Choosing $\gamma=1$, for every $(x_1, x_2) \in K$ and every $a \in [-4.5, 7]$, one has 
$\sin (x_1 + x_2)  \geq  -\frac{1}{2} (x_1 + x_2-a)^2 +\cos(a) (x_1 + x_2 - a) + \sin (a)$.
\if{
\begin{figure}[!ht]
\centering
\includegraphics[scale=0.7]{samp3d.pdf}
\caption{Semialgebraic under-approximations for $(x_1, x_2) \mapsto \sin(x_1 + x_2)$}	\label{fig:samp3d}
\end{figure}
}\fi
\end{example}

\subsection{Nonlinear Templates}
\label{sec:templates_intro}

The non-linear template method is a refinement of polyhedral based
methods in static analysis~\cite{Sankaranarayana+others/05/Scalable}.
It can also be closely related to the non-linear 
extension~\cite{adjegaubertgoubault11} of the template method and
to the class of affine relaxation methods~\cite{Messine_extensionsof}.

Templates allow one to determine invariants of programs
by considering parametric families of subsets of $\R^n$ of the form
\(
S(\alpha) = \{\xb\mid w_i(\xb) \leq \alpha_i, \; 1\leq i\leq p\},
\) 
where the vector $\alpha\in \mathbb{R}^p$ is the parameter,
and $w_1,\dots,w_p$ (the template) are fixed 
possibly non-linear functions, tailored 
to the program characteristics.

The nonlinear template method yields a tradeoff between the coarse bounds of interval calculus and the tighter bounds obtained with high-degree polynomial approximation (see Remark~\ref{rk:templates}). On the one hand, templates take into account the correlations between the different variables. On the other hand, instead of increasing the degree of the approximation, one may increase the number of functions in the template. 

\begin{remark}
\label{rk:templates}
Notice that by taking a trivial template (bound constraints, \ie{}, 
functions of the form $\pm x_i$), the template method specializes to a version of interval calculus, in which
bounds are derived by SOS techniques. 
The standard Taylor (\resp{}Chebyshev) approximations of transcendental functions can also be retrieved by instantiating some of the $w_i$ to degree-$d$ Taylor polynomials (\resp{}best uniform degree-$d$ polynomials).
\end{remark}

The max-plus basis
method introduced in Sect.~\ref{sec:maxplusbasis} is equivalent to the approximation of the epigraph
of a function by a set $S(\alpha)$. This method involves the approximation from below of a function $f$ in $n$ variables
by a supremum
$f \gtrapprox g:=\sup_{1\leq i \leq p} \lambda_i + w_i $. 
The functions $w_i$ are fixed in advance, or dynamically
adapted by exploiting the problem structure. The parameters
$\lambda_i$ are degrees of freedom. 

The template method consists in propagating approximations of the
set of reachable values of the variables of a program
by sets of the form $S(\alpha)$.  
The non-linear template and max-plus approximation methods
are somehow related. Indeed, the $0$-level set of $g$, $\{\xb \mid g(\xb) \leq 0\}$, is nothing but
$S(-\lambda)$, so templates can be recovered
from max-plus approximations and vice versa.
The functions $w_i$ are usually required to be 
quadratic polynomials, $w_i(\xb)  = q_i^T \xb + \frac{1}{2} \xb^T A_i \xb$,
where $q_i\in \R^n$ and $A_i$ is a symmetric matrix. A
basic choice is  $A_i=-\gamma I_n$, where $\gamma$ is a fixed constant. Then, the parameters $q$ remain the only degrees of freedom.

\section{The Nonlinear Template Optimization Algorithm}
\label{sec:template_algo}

Here we explain how to combine semialgebraic optimization techniques with approximation tools for univariate or semialgebraic functions. Let us consider an instance of Problem~\eqref{eq:f}. We assimilate the objective function $f$ with its abstract syntax tree $t$. We assume that the leaves of $t$ are semialgebraic functions in the set $\setA$ and other nodes are univariate transcendental functions ($\arctan$, \etc{}) or basic operations ($+$, $\times$, $-$, $/$). For the sake of the simplicity, we suppose that each univariate transcendental function is monotonic. 

\subsection{A Semialgebraic Template Approximation Algorithm}
\label{sec:template_approx}

The auxiliary algorithm $\templateapprox$ is presented in Fig.~\ref{alg:template_approx}. 

Given an abstract syntax tree $t$, a semialgebraic set $K :=  \{ \xb \in \R^n : g_1(\xb) \geq 0,\dots, g_m(\xb) \geq 0\}$, an SOS relaxation order $k$ and a precision $\precision$ which can be either a finite sequence $s$ of points $\xb_1,\dots,\xb_p \in K$ or a polynomial approximation degree $d$, the algorithm  $\templateapprox$ computes a lower  bound $m$ (\resp{}upper bound $M$) of $t$ over $K$ and an under-approximation $t^-$ (\resp{}an over-approximation $t^+$) of $t$ by means of semialgebraic functions. We assume that the semialgebraic set $K :=  \{ \xb \in \R^n : g_1(\xb) \geq 0,\dots, g_m(\xb) \geq 0\}$ is contained in a box of $\R^n$.

When $t \in \setA$ (Line~\lineref{line:begin_template}), it suffices to set $t^- = t^+ := t$.

When the root of $t$ is a binary operation whose arguments are two children $c_1$ and $c_2$, we apply recursively $\templateapprox$ to each child and get semialgebraic under-approximations $c_{1}^-$, $c_{2}^-$ and
over-approximations $c_{1}^+$, $c_{2}^+$. Then, we obtain semialgebraic approximations of $t$ by using the semialgebraic arithmetic procedure $\composebop$ (the rules are analogous with interval calculus). 

When $t$ corresponds to the composition of a transcendental (unary) function $r$ with a child $c$, lower and upper bounds $m_{c}$ and $M_{c}$ are recursively obtained (Line~\lineref{line:samp_template}), as
well as semialgebraic approximations $c^-$ and $c^+$. Then we define $I := [m_c, M_c]$ and apply the function $\uopapprox$ to get approximations $r^-$ and $r^+$ of $r$ over $I$. The parameters of $\uopapprox$ are the univariate function $r$, the precision $p$ of the approximation and the closed interval $I$ where $r$ must be approximated.
We shall need to consider various schemes for $\uopapprox$:
\begin{enumerate}
\item A classical one is the approximation of univariate functions with best uniform polynomials of increasing degrees through Remez algorithm.
In practice, we use the function $\remez$ available in the $\sollya$ tool~\cite{ChevillardJoldesLauter2010}. In this case, the precision is the degree-$d$ of the minimax polynomial approximation. When the algorithm converges and returns a degree-$d$ polynomial $f_d$, then a numerical approximation of the infinity norm of the error function $(r - f_d)$ on the interval $I$ can be obtained ($\infnorm$ routine from $\sollya$). 
\item An alternative approach is to compute maxplus approximation using the semiconvexity properties of $r$ on the interval $I$ (see Sect.~\ref{sec:samp_lsc}). Doing so, one bounds $r$ from below with a function $r^-$ being a supremum of parabola as well as from above with a function $r^+$ being a infimum of parabola.  
In this case, the precision is determined by certain sets $s$ of points, which also parametrize the approximations $r^-$ and $r^+$ (see e.g. the right hand side of~\eqref{eq:max_par} for an example of under-approximation $r^-$).
\end{enumerate}
The approximations $r^-$ and $r^+$ are composed with $c^-$ and $c^+$ ($\composeapprox$ function at Line \lineref{line:compose_template}) to obtain an under-approximation $t^-$ (resp. over-approximation $t^+$) of $t$. 
Notice that the behavior of $\composeapprox$ depends on the monotonicity properties of $r$.
\begin{figure}[!t]
\begin{algorithmic}[1]                    
\Require tree $t$, semialgebraic $K$, semidefinite relaxation order $k$, precision $\precision$
\Ensure lower bound $m$,  upper bound $M$, lower semialgebraic approximation $t_2^-$, upper semialgebraic approximation $t_2^+$  
	\If {$t \in \setA$}\label{line:begin_template} $t^- := t$, $t^+ := t$
	\ElsIf {$\bop := \astroot (t)$ is a binary operation with children $c_1$ and $c_2$}
		\State $m_i, M_i, c_{i}^-, c_{i}^+ := \templateapproxfun{c_i}{K}{k}{\precision}$ for $i \in \{1,2\}$
	    \State $I_{2} := [m_2, M_2]$
		\State $t^-, t^+ := \composebop (c_{1}^-, c_{1}^+, c_{2}^-, c_{2}^+, \bop, I_{2})$ \label{line:samp_template1}
	\ElsIf {$r := \astroot (t) \in \setD$ with child $c$}	
		\State $m_c$, $M_c$, $c^-$, $c^+ := \templateapproxfun{c}{K}{k}{\precision}$ \label{line:samp_template}
		\State $I := [m_{c}, M_{c}]$
		\State $r^-, r^+ := \uopapprox (r, I, c, \precision)$
		\State $t^-, t^+ :=  \composeapprox (r, r^-, r^+, I, c^-, c^+)$\label{line:compose_template}
	\EndIf\label{line:end_template}
	\State $t_2^- := \saapproxfun {t}{K}{k}{\precision}{t^-}, t_2^+ := -\saapproxfun {t}{K}{k}{\precision}{-t^+}$\label{line:saapprox_template}
	\State \Return $\minsa(t_2^-, K, k)$, $\maxsa(t_2^+, K,  k)$, $t_2^-$, $t_2^+$
\end{algorithmic}
\caption{$\templateapprox$}
\label{alg:template_approx}
\end{figure}
\subsection{Reducing the Complexity of Semialgebraic approximations}
The semialgebraic approximations previously computed are used to determine lower and upper bounds of the function associated with the tree $t$, at each step of the induction. The bounds are obtained by calling the functions $\minsa$ and $\maxsa$ respectively, which reduce the semialgebraic optimization problems to polynomial optimization problems by introducing extra lifting variables (see Section~\ref{sec:sa}).
However, the complexity of solving the SOS relaxations can grow significantly because of the number $\lift$ of lifting variables. If $k$ denotes the relaxation order, the corresponding SOS problem $Q_k$ indeed involves linear matrix inequalities of size $\binom{n +\lift + k}{k}$ over $\binom{n +\lift + 2 k}{2 k}$ variables. 
The complexity of the semialgebraic approximations is controlled with the function $\saapprox$ (Line~\lineref{line:saapprox_template}), when the number of lifting variables exceeds a user-defined threshold value $\lift^{\max}$. Consequently, this is crucial to control the number of lifting variables, or equivalently, the complexity of the semialgebraic approximations. For this purpose, we introduce two approximation schemes.
The first one is presented in Sect.~\ref{sec:quadr_templates}. It allows to compute approximations for some sub-components of the tree $t$ (or its under-approximation $t^-$) by means of suprema/infima of quadratic functions.
An alternative approach is to approximate these sub-components with degree-$d$ polynomial under-approximations, using the semidefinite relaxation described in Sect.~\ref{sec:sdp_L1}.
%
\subsubsection{Multivariate Maxplus Quadratic Templates}
\label{sec:quadr_templates}
Let $K \subset \R^n$ be a compact semialgebraic set and $f : K \to \R$ be a multivariate nonlinear function.
We consider the vector space $\Sn$ of real symmetric $n \times
n$ matrices. Given a matrix $M \in \Sn$, let $\lmax(M)$ (resp. $\lmin(M)$) be the maximum (\resp{}minimum) eigenvalue of $M$. In the sequel, we will often refer to the quadratic polynomial defined below. 
\begin{definition}
Let $\xb_c \in K$. The quadratic polynomial $f_{\xc, \lambda'}$ is given by:
\begin{align}
\label{eq:fxc}
    f_{\xb_c, \lambda'} : K \longrightarrow{} & {} \R \nonumber\\
    x \longmapsto{} & {} f(\xb_c) + \Df (\xb_c) \, (\xb - \xb_c)  \\
    & + \frac{1}{2}(\xb - \xb_c)^{T} \Hf (\xb_c)  (\xb - \xb_c) \nonumber\\
    & + \frac{1}{2} \lambda' \ltwonorm{\xb - \xb_c}^2 \enspace,  \nonumber
\end{align}
with,
\begin{equation}
\label{eq:lambdaK}
\lambda' \leq \lambda := \min_{\xb \in K} \{ {\lmin ( \Hf (\xb) - \Hf (\xb_c)  )} \}\enspace.
\end{equation}
\end{definition}
The quadratic polynomial $f_{\xc, \lambda'}$ is an under-approximation of $f$ on the set $K$:
\begin{lemma}
\label{th:tm_quadr_templates}
$\forall \xb \in K, \ f(\xb) \geq f_{\xb_c, \lambda'}$.
\end{lemma}
\begin{proof}
It comes from the first order Taylor expansion with the integral form for the remainder and the definition of the minimal eigenvalue.
\qed
\end{proof}
\begin{definition}
\label{def:spectralradius}
Given a symmetric real-valued matrix $M \in \Sn$,  the spectral radius of $M$ is given by $\rho(M) := \max (\lmax(M), -\lmin(M) )$.
\end{definition}
In the sequel, we use the following inequality:
\begin{equation}
\label{eq:specrad_ineq}
\rho(M) \leq \lonenorm{M} := \max_{x \neq 0} \dfrac{\lonenorm{M x} }{\lonenorm{x}} \enspace.
\end{equation}
Now, we explain how to approximate $f$ from below over $K$ with quadratic polynomials as in~\eqref{eq:fxc}. 
To approximate from below the value of $\lambda$, we determine an interval matrix
$\widetilde{\Hf} := ([\dijlo, 
\dijup ])_{1 \leq i, j \leq n}$, containing coarse bounds of the Hessian difference entries, using interval arithmetic. 
Let consider the following interval matrix minimal eigenvalue problem:
\begin{equation}
\label{eq:lminhessian}
\lambda' := \lmin (\widetilde{\Hf}) \enspace.
\end{equation}
%
%
Different approximations of $\lambda$ can be considered.
\subsubsection*{Tight lower bound of $\lambda$}
For each interval $[\dijlo, \dijup]$, we define the symmetric matrix $B$:
\[B_{i j} := \max \{\mid \dijlo \mid, \mid \dijup \mid \}, \enspace 1 \leq i, j \leq n \enspace.\]
Let $\mathcal{S}^n$ be the set of diagonal matrices of sign:
\[\mathcal{S}^n := \{\diag (s_1,\dots,s_n), \ s_1 = \pm 1,\dots\,s_n = \pm 1 \}\enspace. \]
Next, one specializes the result in~\cite[Theorem 2.1]{springerlink:10.1007/s10957-008-9423-1} for robust optimization procedure with reduced vertex set.
\begin{lemma}
\label{th:lminhessian} 
The robust interval SDP Problem~\eqref{eq:lminhessian} is equivalent to 
the following single variable SDP:
\[	
\left\{
\begin{array}{lrl}
\min\limits_t & -t \enspace, &\\
\st & -t I  - S  B  S & \succeq 0 \enspace, \\
& S & = 
\begin{pmatrix}
  1 & 0\\
  0 & S' 
\end{pmatrix}
, \quad \forall S' \in \mathcal{S}^{n - 1} \enspace.
\end{array}\right.
\]			
Let $\lambda_1'$ be the solution of this single variable SDP. Then, $\lambda_1' \leq \lambda.$
\end{lemma}
However, solving the semidefinite program given in Lemma~\ref{th:lminhessian} introduces a subset of sign matrices of cardinal $2^{n - 1}$, thus reduces the problem to a manageable size only if $n$ is small. 
\subsubsection*{Coarse lower bound of $\lambda$}
Here, one writes $\hessI := X + Y$, where $X$ and $Y$ are defined as follows:
\[ X_{i j} := \Bigl[\mI{\dijlo}{\dijup}, \mI{\dijlo}{\dijup} \Bigr], \quad  Y_{i j} :=\Bigl[ -\dI{\dijup}{\dijlo}, \dI{\dijup}{\dijlo}\Bigr] \enspace.\]
\begin{proposition}
Define $\lambda_2' := \lmin (X) - \max_{1 \leq i \leq n} \Bigl\{ \sum_{j = 1}^n \dI{\dijup}{\dijlo} \Bigr\}$. Then, $\lambda_2' \leq \lambda.$
\end{proposition}
\begin{proof}
By concavity and homogeneity of the $\lmin$ function, one has:
\begin{equation}
\lmin(X + Y) \geq \lmin (X) + \lmin (Y) = \lmin (X) - \lmax (-Y)\enspace.
\end{equation}
Using Proposition~\ref{eq:specrad_ineq} yields $\lmax(-Y) \leq \max_{1 \leq i \leq n} \Bigl\{ \sum_{j = 1}^n \dI{\dijup}{\dijlo} \Bigr\}$.
\if{
\begin{equation}
\lmax(-Y) \leq \max_{1 \leq i \leq n} \Bigl\{ \sum_{j = 1}^n \dI{\dijup}{\dijlo} \Bigr\}\enspace.
\end{equation}
}\fi
\qed
\end{proof}
The matrix $X$ is real valued and symmetric matrix, thus one can compute its minimal eigenvalue with the classical semidefinite program: $\min \{ -t : X  \succeq t I \}.$
\if{
\[	
\left\{
\begin{array}{ll}
\min & -t \\
\st & X -t I \succeq 0 \enspace.
\end{array}\right.
\]
}\fi
Finally, we can compute a coarse certified lower bound $\lambda_2'$ of $\lambda$ with a procedure which is polynomial in $n$. 

Now we describe the procedure $\saapprox$ in case of using the multivariate maxplus quadratic templates defined in this section, assuming that the precision $p$ is determined by a sequence of points $s$.
For each point $\xb_c$ of the sequence $s$, a sub-routine returns the polynomial defined in~\eqref{eq:fxc}. In particular, one has $f_{\xb_c, 1} := \tmalgofun {f} {\xb_c} {\lambda_1'}$, which is built with the tight eigenvalue approximation $\lambda_1'$.  Similarly, 
$f_{\xb_c, 2}$ is defined with the coarse eigenvalue approximation $\lambda_2'$. Since each $f_{\xb_c, i}$ does not necessarily approximate $f$ from below, we determine a lower bound $m_c$ of the function $t^- - f_{\xb_c, i}$, which ensures that  $t_2^- := \max\limits_{\xb_c \in s} \{f_{\xb_c, i} + m_c \}$ is a valid lower approximation of $f$.

\subsubsection{Polynomial under-approximations for Semialgebraic Functions}
\label{sec:l1_under}

Given a box $K \subset \R^n$, we consider a semialgebraic sub-component $\fsa : K \to \R$ of the abstract syntax tree of $f$.
A common way to represent $\fsa$ is to use its semialgebraic lifting, which leads to solve semialgebraic optimization problems with a possibly large number of lifting variables $\lift$.
One way to reduce this number is to approximate $\fsa$ from below with a degree-$d$ polynomial $h_d$, which should involve less variables than $\lift$. This section describes how to obtain such an $h_d$, which has the property to minimize the $\lone$ norm of the difference $(\fsa - h)$, over all degree-$d$ polynomial under-approximations $h$ of $\fsa$. We exploit a technique of Lasserre and Thanh~\cite{cvxunder}, who showed how to obtain convex under-approximations of polynomials. 
Here, we derive a similar hierarchy of SOS relaxations, whose optimal solutions are the best (for the $\lone$ norm) degree-$d$ (but possibly non convex) polynomial under-approximations of $t$ on $K$. We assume without loss of generality that 
$K$ is the unit ball $[0, 1]^n$. By comparison with~\cite{cvxunder}, 
the main difference is that the input is a semialgebraic function, rather than a polynomial.

\paragraph{Best polynomial under-approximations of semialgebraic functions for the $\lone$ norm.}

Let $\fsa : [0, 1]^n \to \R$ be a semialgebraic component of $f$ and $\lambda_n$ be the standard Lebesgue measure on $\R^n$, which is normalized so
that $\lambda_n([0, 1]^n)=1$. Define $g_1 := x_1 (1 - x_1), \dots, g_n := x_n (1 - x_n)$.
The function $\fsa$ has a basic semialgebraic lifting, thus there exist $p, s \in \N$, polynomials $g_{n + 1}, \dots , g_{n + s} \in \R[\xb, z_1,\dots,z_p]$ and a basic semialgebraic set $\Kpop$ defined by:
\[\Kpop :=  \{ (\xb, \zb) \in \R^{n+p} : g_{1}(\xb, \zb) \geq 0,\dots, g_{m}(\xb, \zb) \geq 0,  g_{m + 1}(\xb, \zb) \geq 0\} \enspace, \]
such that the graph $\Psi_{\fsa}$ satisfies:
\[\Psi_{\fsa} := \{ (\xb, \fsa(\xb)) : \xb \in K \} = \{ (\xb, z_p) : (\xb,  \zb) \in \Kpop\} \enspace, \]
with $m := n + s$ and $g_{m + 1} := M - \ltwonorm{\zb}^2$, for some positive constant $M$ obtained  by adding bound constraints over the lifting variables $\zb$ (to ensure that the module $QM(\Kpop)$ is Archimedean).
Define the polynomial $\fpop(\xb, \zb) := z_p$ and the total number of variables $\npop := n + p$. 

Consider the following optimization problem with optimal value $m_d$:

\[
(\Psa) \left\{			
\begin{array}{ll}
\min\limits_{h \in \R_d[\xb]  } & \displaystyle\int_K (\fsa - h) d \lambda_n \\	
\text{s.t.} & \fsa - h \geq 0  \textrm{ on }  K \enspace.
\end{array} \right.
\]

\begin{lemma}
\label{th:psa}
Problem $(\Psa)$ has a degree-$d$ polynomial minimizer $h_d$.
\end{lemma}

For a proof, see Appendix~\ref{pr:psa}. Now, define  $QM(\Kpop)$ to be the quadratic module associated with $g_1, \dots, g_{m + 1}$. As a consequence of Putinar's Positivstellensatz for Archimedean quadratic modules~\cite{putinar1993positive}, the optimal solution $h_d$ of $(\Psa)$ is a maximizer of the following problem:

\[
(P_d) \left\{			
\begin{array}{ll}
\max\limits_{h \in \R_d[\xb]  } & \displaystyle\int_{[0, 1]^n} h \ d \lambda_n \\	
\text{s.t.} & (\fpop - h) \in QM(\Kpop) \enspace.
\end{array} \right.
\]

Let $\mu_d$ be the optimal value of $(P_d)$. Then, one has $m_d = \int_K \fsa \ d \lambda  - \mu_d$. 

\paragraph{Convergent hierarchy of SOS relaxations.}
\label{sec:sdp_L1}

We write $h = \sum_{\alphab \in \mons{d}{n}} h_{\alphab} \xb^{\alphab}$, with $\mons{d}{n} := \{ \alphab \in \N^n : \sum_{i=1}^n \alpha_i \leq d \}$.
Let $\omegatilde{0} :=  \lceil (\deg g_0) / 2 \rceil, \dots, \omegatilde{m+1} :=  \lceil (\deg g_{m + 1}) / 2 \rceil$ and
define $k_0 := \max \{\lceil d / 2 \rceil, \lceil (\deg \fpop) / 2 \rceil, \omegatilde{0}, \dots, \omegatilde{m+1} \}$.
Now, consider the following SOS relaxation $(P_{d k})$ of $(P_d)$, with optimal value $\mu_{d k}$:
\[
(P_{d k}) \left\{			
\begin{array}{ll}
\max\limits_{h \in \R_d[\xb], \sigma_j} & \sum_{\alphab \in \mons{d}{n}} h_{\alphab} \gamma_{\alphab} \\			
\text{s.t.} & \fpop (\xb, \zb) = h (\xb) + \displaystyle\sum_{j=0}^{m+1} \sigma_j(\xb, \zb) g_j (\xb, \zb), \enspace \forall (\xb, \zb) \enspace, \\
            & \sigma_j \in \Sigma_{k - \omegatilde{j}}[\xb, \zb], \enspace 0 \leq j \leq m + 1 \enspace,
\end{array} \right.
\]
with $k \geq k_0$ and $\gamma_{\alphab} := \int_{[0, 1]^n} \xb^{\alphab} d \xb$ for all $\alphab \in \mons{d}{n}$.
%
This problem is an SOS program with variables $(h_d, \sigma_0, \dots, \sigma_{m + 1})$.
Let $m_d$ be the optimal value of Problem~$(\Psa)$. As in~\cite{cvxunder}, the optimal value of the SOS relaxation $(P_{d k})$ can become as close as desired to $m_d - \fsa^*$.
\begin{theorem}
\label{th:cvg_sos_l1}
The sequence $(\int_K \fsa d \lambda - \mu_{d k})_{k \geq k_0}$ is non-increasing and converges to $m_d$. Moreover, if $h_{d k}$ is a maximizer of $(P_{d k})$, then the sequence $(\lonenorm{\fsa - h_{d k}})_{k \geq k_0}$ is non-increasing and converges to $m_d$. Furthermore, any accumulation point of the sequence $(h_{d k})_{k \geq k_0}$ is an optimal solution of Problem~$(\Psa)$.
\end{theorem}
\begin{proof}
\label{pr:cvg_sos_l1}
The proof is analogous with~\cite[Theorem~3.3]{cvxunder}. 
\qed
\end{proof}
\paragraph{Numerical experiments.}

We present the numerical results obtained when computing the best degree-$d$ polynomial under-approximations of semialgebraic functions for the $\lone$ norm, using the techniques presented in Sect.~\ref{sec:sdp_L1}. The sequence of lower bounds $(\mu_{d k})$ is computed by solving the SOS relaxations $(P_{d k})$. The ``tightness'' score $\lonenorm{\fsa - h_{d k}}$ evaluates the quality of the approximation $h_{d k}$, together with its lower bound $\mu_{d k}$.
%
\begin{example}
\label{ex:9922_sa_l1}
In Example~\ref{ex:9922_sa}, we obtained lower bounds for the semialgebraic function  $\fsa := \frac{\partial_4 \Delta \xb }{\sqrt{4 x_1 \Delta \xb}}$, using two lifting variables.
However, when solving inequalities involving $\fsa$, one would like to solve POP that do not necessarily include these two lifting variables and the associated constraints. 
\begin{table}[!ht]
\begin{center}
\caption{Comparing the tightness score $\lonenorm{\fsa - h_{d k}}$ and $\mu_{d k}$ for various values of $d$ and $k$}
\footnotesize
\begin{tabular}{l|l|c|r}
\hline
$d$ & $k$ & Upper bound of $\lonenorm{\fsa - h_{d k}}$ & $\mu_{d k}$\\
\hline
\multirow{2}{*}{2} & 2 & 0.8024 &  -1.171 \\
                   & 3 & 0.3709 & -0.4479 \\
\hline
\hline
\multirow{2}{*}{4} & 2 & 1.617 &  -1.056 \\
                   & 3 & 0.1766 &  -0.4493 \\
\hline
6                  & 3 & 0.08826 & -0.4471 \\
\hline
\end{tabular}
\label{table:l1_under}
\end{center}
\end{table}
Table~\ref{table:l1_under} displays the tightness scores and the lower bounds of the approximations obtained for various values of the approximation degree $d$ and the relaxation order $k$. Notice that $\mu_{d k}$ only bounds from below the actual infimum $h_{d k}^*$ of the under-approximation $h_{d k}$. It requires a few seconds to compute approximations at $k = 2$ against 10 minutes at $k = 3$, but one shall consider to take advantage of replacing $\fsa$ by its approximation $h_{6 3}$ to solve more complex POP.
\end{example}
As an alternative to the method proposed in Sect.~\ref{sec:quadr_templates}, the procedure $\saapprox$ can return the polynomial under-approximation $h_{d k}$ while solving the SOS relaxation $(P_{d k})$.

\if{
\begin{example}
\label{ex:l1_under_3d}
To illustrate the method, we consider the six variables function $\radx$ issued from Flyspeck and its projection $\radx_2$ with respect to the first two coordinates $(x_1, x_2)$ on the box $[4, 8]^2$ (we instantiated the remaining variables by the constant value $8$):
\[\radxtwo : (x_1, x_2) \mapsto  \frac{- 64 x_1^2 + 128 x_1 x_2 + 1024 x_1 - 64 x_2^2   + 1024 x_2 - 4096} 
{- 8 x_1^2 + 8 x_1 x_2 + 128 x_1 - 8 x_2^2 + 128 x_2 - 512} \enspace.\]

In Figure~\ref{fig:l1_under_3d} is displayed $\radxtwo$ after scaling on $[0, 1]^2$ (the red surface), as well as the linear under-approximation $h_{1 3}$ (blue surface) and quadratic under-approximation $h_{2 3}$ (green surface). Both under-approximations are obtained at the third relaxation order.
\end{example}

\begin{figure}[!ht]
\centering
\includegraphics[scale=1]{l1_under_3d.pdf}
\caption[Linear and Quadratic Polynomial under-approximations for $\radxtwo$]{Linear and Quadratic Polynomial Under-approximations for the $\radxtwo$  function}	\label{fig:l1_under_3d}
\end{figure}
}\fi

\subsection{A Semialgebraic Template Optimization Algorithm}

Our main optimization algorithm $\templateoptim$ is an iterative procedure which relies on $\templateapprox$. At each iteration step, the global precision parameter $p \in \precset$ is updated dynamically.
A convenient way to express the refinement of the precision,
for the general nonlinear template approximation scheme (see Fig.~\ref{alg:template_approx}), is to use the vocabulary of nets. We recall the following definitions, using~\cite{nagata1974modern}:

\begin{definition}
\label{def:dirset}
A directed set is a set $\dirset$ with a relation $\leq$ which is reflexive, transitive and directed, \ie{}for each $a, b \in \dirset$, there exists some $c \in \dirset$ such that $a \leq c$ and $b \leq c$.
\end{definition}

\begin{definition}
\label{def:net}
A net in a set $X$ is a map $\lambda : \dirset \to X$. If $X$ is a topological space, we say that the net $\lambda$ converges to $x \in X$ and write $\lambda \to x$ if and only if for every neighborhood $U$ of $x$, there exists some tail $\Lambda := \{\lambda(c) : d \leq c \in \dirset\}$ such that $\Lambda \subseteq U$.
\end{definition}

We represent the precision $p$ by an element of a directed set $\precset$. When using minimax polynomial approximations to approximate an univariate function on a given interval $I$, the sequence of approximation degrees defines the net. For the maxplus approximations, the net is the set of finite subsets of $I$. 

Let $c_1, \dots, c_l$ be the components of the tree $t$, on which one calls approximation algorithms with respective precisions $p_1 \in \precset_1, \dots, p_l \in \precset_l$. Let $\precset = \precset_1 \times \dots \times \precset_l$ be the set of precisions, ordered with the product order. 

Our main optimization algorithm $\templateoptim$, relies on $\templateapprox$ and updates the global precision parameter $p \in \precset$ dynamically at each step of an iteration procedure (Line~\lineref{line:randeval}). 

\begin{figure}[!t]                    
\begin{algorithmic}[1]                  
\Require abstract syntax tree $t$, semialgebraic set $K$,  $\itermax$ (optional argument), precision $p$
\Ensure lower bound $m$
 	\State $s := [\argmin (\randeval (t)) ]$ \Comment{$s \in K$} \label{line:randeval}
 	\State $m := -\infty$
 	\State $\iter := 0$
	\While {$\iter \leq \itermax$}\label{line:begin_optim}
	\State Choose an SOS relaxation order $k \geq k_0$\label{sdp_order}
	\State $m, M, t^-, t^+ := \templateapproxfun{t}{K}{k}{\precision}$ \label{line:approx}
	\State $\Xopt  := \guessargmin (t^-)$ \Comment{$t^- (\Xopt) \simeq m$}\label{line:argmin} 	 		
	\State $s := s \cup \{ \Xopt \}$ \label{line:add_control_point}
	\State $\precision := \updateprecision (\precision, \Xopt)$
	\State $\iter := \iter + 1$
	\EndWhile\label{line:end_optim}		
	\State \Return $m$, $\Xopt$
\end{algorithmic}			
\caption{$\templateoptim$ : Template Optimization Algorithm}   \label{alg:templateoptim}
\end{figure}
Now we describe our main semialgebraic optimization algorithm $\optim$ (see Figure~\ref{alg:templateoptim}). Given an abstract syntax tree $t$ and a compact semialgebraic set $K$ this algorithm returns a lower bound $m$ of $t$ using semialgebraic minimax approximations computed recursively with $\templateapprox$. The relaxation order $k$ (Line~\lineref{sdp_order}) is a parameter of the semialgebraic optimization functions $\minsa$ (as well as $\maxsa$) and $\saapprox$.

Assuming that $K$ is described by polynomial inequalities $g_1(\xb) \geq 0, \dots, g_m(\xb) \geq 0$. 
Then, the relaxation order must be at least $k_0 := \max_{1 \leq j \leq m} \{ \lceil \deg (g_j) / 2\rceil) \}$. In practice, we solve semialgebraic optimization problems with the second or third SOS Lasserre's relaxation and take $k = k_0$.
At the beginning, the set of points consists of a single point of the box $K$. This point is chosen so that it minimizes the value of the function associated to the tree $t$ among a set of random points (Line~\lineref{line:randeval}). 
Then, at each iteration of the loop from Lines~\lineref{line:begin_optim} to~\lineref{line:end_optim}, the auxiliary function $\templateapprox$ is called to compute a lower bound $m$ of the function $t$ (Line~\ref{line:approx}), using the approximations $t^-$ and $t^+$.
At Line~\lineref{line:argmin}, a minimizer candidate $\Xopt$ of the under-approximation tree $t^-$ is computed. It is obtained by projecting a solution $\Xsdp$ of the SOS relaxation $Q_k$ of Section~\ref{sec:sos} on the coordinates representing the first order moments, following~\citep[Theorem~4.2]{DBLP:journals/siamjo/Lasserre01}. 
However, the projection may not belong to $K$ when the relaxation order $k$ is not large enough. This is why tools like $\sparsepop$ use local optimization solver in a post-processing step to provide a point in $K$ which may not be a global minimizer. In any case, $\Xopt$ is then added to the set of points (Line~\ref{line:add_control_point}). Alternatively, if we are only interested in determining whether the infimum of $t$ over $K$ is nonnegative (Problem~\eqref{eq:ineq}), the loop can be stopped as soon as $m \geq 0$.
\if{

When using maxplus approximation, the precision is a finite set of points. At the beginning, this set consists of a single point of the semialgebraic set $K$, chosen randomly (Line~\lineref{line:randeval}).
Until the maximal number of iterations is reached, the auxiliary function $\templateapprox$ is called to compute a lower bound $m$ of the function $t$, using the approximations $t^-$ and $t^+$.

A minimizer candidate $\Xopt$ of the under-approximation tree $t^-$ is computed as in Fig.~\ref{alg:poptemplateoptim} and added to the set of points. 
When we are only interested in determining whether the infimum of $t$ over $K$ is nonnegative (Problem~\eqref{eq:ineq}), the stopping criterion for the iterative procedure is $m \geq 0$.
}\fi

By comparison, when using minimax approximations, the stopping criterion is the maximal precision corresponding to a minimax polynomial approximation degree. This maximal degree $\dmax$ shall be selected after consideration of the computational power available since one may need to solve SOS relaxations involving $O(\dmax^n)$ variables with matrices of size $O({\lceil \dmax / 2 \rceil}^n)$.

\begin{example}[Lemma$_{9922699028}$ Flyspeck]
\label{ex:9922_samp}
We continue Example~\ref{ex:9922_sa}. Since we computed lower and upper
bounds ($m$ and $M$) for $\fsa := \frac{\partial_4 \Delta \xb}{\sqrt{ 4 x_1 \Delta \xb }}$, we know that the $\fsa$ argument of $\arctan$
lies in $I := [m, M]$. We describe three iterations of the
algorithm $\templateoptim$ while using maxplus approximation for $\uopapprox$. 
  
\begin{enumerate}
\setcounter{enumi}{-1} 
\item Multiple evaluations of $f$ return a set of values and we obtain
  a first minimizer guess $\xb_1 :=  \argmin (\randeval (f))$ corresponding to the minimal value of the set. One has  $\xopt{1}{4.8684}{4.0987}{4.0987} {7.8859}{4.0987}{4.0987}$.
  
\item We compute $a_1 := \fsa (\xb_1) = 0.3962$ and define
  $\parab_{a_1}^-(x) := -\frac{1}{2} (x-a_1)^2 + \frac{1}{1 + a_1^2} (x - a_1) + \arctan (a_1)$. Finally, we obtain
  $ m_1 \leq \min_{\xb \in K} \{l(\xb) + \parab_{a_1}^- (\fsa (\xb)) \}$.
  For $k=2$, one has $m_1 = -0.2816 < 0$ and a new minimizer $\xb_2$. 
  
\item We get $a_2 := \fsa (\xb_2)$ and
  $ m_2 \leq \min_{\xb \in K} \{l(\xb) + \max_{1 \leq i \leq 2 } \{ \parab_{a_i}^- (\fsa (\xb))\} \}$. For $k=2$, we get $m_2 = -0.0442 < 0$ and 
 $\xb_3$.  
  
\item We get $a_3 := \fsa (\xb_3)$,  $\parab_{a_3}^-$ and
  $ m_3 \leq \min_{\xb \in K} \{l(\xb) + \max_{1 \leq i \leq 3} \{ \parab_{a_i}^- (\fsa (\xb))\} \}$. For $k=2$, we obtain $m_3 = -0.0337 < 0$ and get a new minimizer $\xb_4$.
\end{enumerate}
\end{example}

\subsection{Convergence of the Nonlinear Template Method}
\label{sec:cvg}

Given an accuracy $\epsilon > 0$, we prove that the objective function $f$ can be uniformly $\epsilon$-approximated over the semialgebraic set $K$ with the algorithm $\templateapprox$ under certain assumptions. 
\begin{assumption}
\label{minimaxapprox_cvg_hyps}
The Archimedean condition holds for the quadratic modules that we consider when solving SOS relaxations.
\end{assumption}

For the sake of simplicity, we assume that the function $\saapprox$ calls the procedure that returns the sequence of best (for the $\lone$ norm) polynomial under-approximations for semialgebraic functions (see Sect.~\ref{sec:l1_under}).
Let the relaxation order $k$ be fixed and $t_p^-$ (\resp{}$t_p^+$) be the under-approximation (\resp{}over-approximation) of $t$ on $K$ obtained with the $\templateapprox$ function at precision $\precision$. 
The limit of a net indexed by $p \in \precset$ is obtained by increasing the precision of each elementary approximation algorithm (either $\uopapprox$ or $\saapprox$) applied to the components of $t$.

\begin{proposition}[Convergence of $\templateapprox$]
\label{th:approx_cvg}
Under Assumption~\ref{minimaxapprox_cvg_hyps}, the nets $(t_{p}^-)_{p}$ and $(t_{p}^+)_{p}$ uniformly converge to $t$ on $K$.
\end{proposition}

For a proof, see Appendix~\ref{pr:approx_cvg}. Given a precision $p$, define $m_p^* := \inf_{\xb \in K} t_p^- $ to be the optimal value  of the under-approximation $t_p^-$ on $K$.
Notice that under Assumption~\ref{minimaxapprox_cvg_hyps}, we can theoretically obtain this optimal value, using Theorem~\ref{th:approx_cvg}.

\begin{corollary}[Convergence of the approximations optimal values]
\label{th:approx_cvg_optim}
Under Assumption~\ref{minimaxapprox_cvg_hyps}, the net $(m_p^*)_{p}$ converges to the infimum $f^*$. 
\end{corollary}

\begin{proof}
Let $\xb_p^*$ be a minimizer of $t_p^-$ on $K$ and note $\xb^*$ one minimizer of $t$ on $K$, then one has $t(\xb^*) = f^*, t_p^-(\xb_p^*) = m_p^*$.
By definition, the following inequalities hold:
\begin{equation}
\label{preq:7}
t_p^-(\xb_p^*) \leq  t_p^-(\xb^*) \leq t(\xb^*) \leq t(\xb_p^*) \enspace.
\end{equation}

Let $\epsilon > 0$ be given. From Proposition~\ref{th:approx_cvg}, there exists a precision $d_0$ such that for all $d \geq d_0$, one has:
$t(\xb^*) - t_p^-(\xb^*) <  \epsilon / 2$ 
and $t(\xb_p^*) - t_p^-(\xb_p^*) <  \epsilon / 2$. 
Thus, applying~\eqref{preq:7} yields $t(\xb^*) - t_p^-(\xb_p^*) < \epsilon$, the desired result.

\qed
\end{proof}

\begin{corollary}[Convergence of $\templateoptim$]
\label{th:approx_cvg_minimizer}
Under Assumption~\ref{minimaxapprox_cvg_hyps}, each limit point of the net of minimizers $(\xb_p^*)_{p}$ is a global minimizer of $t$ over $K$.
\end{corollary}

For a proof, see Appendix~\ref{pr:approx_cvg_minimizer}.

\section{Numerical Results of the Nonlinear Template Method}
\label{sec:benchs}

We now present some numerical test results by applying the semialgebraic minimax optimization method to examples from the global optimization literature (see Appendix~\ref{sec:appa}), as well as inequalities from the Flyspeck project.
The nonlinear template method is implemented as a software package, called \texttt{NLCertify}, written in $\ocaml$ and interfaced with the $\sollya$ tool. For more details about this tool, we refer to the software web-page\footnote{\url{http://nl-certify.forge.ocamlcore.org/}} as well as to the dedicated publication~\cite{icms14}.

For each problem presented in Table~\ref{table:templates_go}, our aim is to certify a lower bound $m$ of a function $f$ on a box $K$. The semialgebraic optimization problems are solved at the SOS relaxation order $k$. When the relaxation gap is too high to certify the requested
bound, then we perform a domain subdivision in order to get tighter bounds:
we divide the maximal width interval of $K$ in two halves to get two sub-boxes $K_1$ and $K_2$ such that $K = K_1 \cup K_2$. We repeat this subdivision procedure,  by applying $\templateoptim$ on a finite set of sub-boxes,
until we succeed to certify that $m$ is a lower bound of $f$.
We note $\nbb$ the total number of sub-boxes generated
by the algorithm.

The algorithm $\templateoptim$ returns more precise bounds by successive updates of the precision $\precision$. For each univariate component $u \in \setD$ of the objective $f$, we note $\#s_u$ the number of points for the maxplus approximations of $u$ and $d_u$ the degree of the minimax approximation of $u$. 

A template-free SOS method coincides with the particular case in which $d_u = 0$ (or $\#s_u = 0$) for each univariate component $u \in \setD$ and $\lift = 0$. We mentioned in~\cite{victorcicm} that this method already outperforms the interval arithmetic solvers. However, it can only be used for problems with a moderate number of variables. The algorithm $\templateoptim$ allows us to overcome this restriction, while keeping a similar performance (or occasionally improving this performance) on medium-size examples. 

The minimax approximation based method is eventually faster
than the maxplus based method for moderate instances. For the example
{\em H3} (\resp{}{\em H6}), the speed-up factor is $2$ when the function
$\exp$ is approximated by a quartic (\resp{}quadratic) minimax polynomial.
On the other hand, notice that reducing the number of lifting variables
allows us to provide more quickly coarse bounds for large-scale instances of the Schwefel problem. We discuss the results appearing in the two last lines of Table~\ref{table:templates_go}.
Without any box subdivision, we can certify a better lower bound $m = -967 n$ with $\lift = 2n$ since our semialgebraic approximation is more precise. However the last lower bound $m = -968 n$ can be computed twice faster by considering only $n$ lifting variables, thus reducing the size of the POP described in Example~\ref{ex:modified_swf}. This indicates that the method is able to avoid the explosion for certain hard sub-classes of problems where a standard
(full lifting) POP formulation would involve a large number of lifting variables.  

\begin{table}[!ht]
\small
\begin{center}
\caption{Numerical results for global optimization examples using $\templateoptim$}
\begin{tabular}{llc|rcccr}
\hline
Pb & $n$ & $m$ & $\precision$ & $\lift$ & $k$ & $\nbb$ & time \\
\hline
\multirow{5}{*}{{\em H3}} & \multirow{5}{*}{$3$} & \multirow{5}{*}{$-3.863$} & $\#s_{\exp} = 3$ & 4 & 2 & 99 & $101 \, s$\\
& & & $d_{\exp} = 0$ & 0 & 1 & 1096 & $247 \, s$\\
& & & $d_{\exp} = 2$ & 4 & 1 & 53 & $132 \, s$\\
& & & $d_{\exp} = 4$ & 4 & 2 & 19 & $57 \, s$\\
& & & $d_{\exp} = 6$ & 4 & 3 & 12 & $101 \, s$\\
\hline
\multirow{3}{*}{{\em H6}} & \multirow{3}{*}{$3$} & \multirow{3}{*}{$-3.33$} & $\#s_{\exp} = 1$ & 6 & 2 & 113 & $102 \, s$\\
& & & $d_{\exp} = 0$ & 0 & 1 & 113 & $45 \, s$\\
& & & $d_{\exp} = 2$ & 4 & 2 & 53 & $51 \, s$\\
\hline
\multirow{6}{*}{{\em MC}} & \multirow{6}{*}{$2$} & \multirow{6}{*}{$-1.92$} & $\#s_{\sin} = 2$ & 4 & 2 & 17 & $1.8 \, s$\\
& & & $d_{\sin} = 0$ & 0 & 1 & 92 & $7.6 \, s$\\
& & & $d_{\sin} = 2$ & 0 & 1 & 8 & $6.3 \, s$\\
& & & $d_{\sin} = 4$ & 0 & 2 & 4 & $3.2 \, s$\\
& & & $d_{\sin} = 6$ & 0 & 3 & 2 & $3 \, s$\\
& & & $d_{\sin} = 8$ & 0 & 4 & 1 & $1.9 \, s$\\
\hline
\multirow{5}{*}{{\em ML}} & \multirow{5}{*}{$10$} & \multirow{5}{*}{$-0.966$} & $\#s_{\cos} = 1$ & 5 & 1 & 5 & $8.2 \, s$\\
& & & $d_{\cos} = 0$ & 0 & 1 & 8 & $6.6 \, s$\\
& & & $d_{\cos} = 2$ & 5 & 1 & 1 & $6.4 \, s$\\
& & & $d_{\cos} = 4$ & 5 & 2 & 1 & $8.1 \, s$\\
\hline
\hline
\multirow{7}{*}{{\em SWF} ($\epsilon = 0$)} & \multirow{3}{*}{$10$} & \multirow{3}{*}{$-430 n$} & $\#s_{\sin} = 6$ & $2 n$ & 2 & 16 & $40 \, s$\\
& & & $d_{\sin} = 0$ & $2 n$ & 1 & 3830 & $129 \, s$\\
& & & $d_{\sin} = 2$ & $2 n$ & 1 & 512 & $2280 \, s$\\
\cline{2-3}
 & \multirow{2}{*}{$10^2$} & \multirow{2}{*}{$-440 n$} & $\#s_{\sin} = 6$ & $2 n$ & 2 & 274 & $6840 \, s$\\
&  &  & $d_{\sin} = 0$ & 0 & 1 & $> 10^4$ & $ > 10^4 \, s$\\
\cline{2-3}
& \multirow{2}{*}{$10^3$} & $-486 n$ & $\#s_{\sin} = 4$ & $2 n$ & 2 & 1 & $450 \, s$\\
&  & $-488 n$ & $\#s_{\sin} = 4$ & $n$ & 2 & 1 & $250 \, s$\\
\hline
\multirow{2}{*}{{\em SWF} ($\epsilon = 1$)} & \multirow{2}{*}{$10^3$} & $-967 n$ & \multirow{2}{*}{$\#s_{\sin} = 2$} & $2n$ & 2 & 1 & $543 \, s$\\
& & $-968 n$&  & $n$ & 2 & 1 & $272 \, s$\\
\hline
\end{tabular}
\label{table:templates_go}
\end{center}
\end{table}
\begin{table}[!ht]
\small
\vspace*{-0.2cm}
\begin{center}
\caption{Results for Flyspeck inequalities using $\templateoptim$ with $n = 6$, $k = 2$, $m = 0$ and $\#s = \#s_{\arctan}$}
\begin{tabular}{lr|lccr}
\hline
Inequality id & $\nt$   & $\precision$ & $\lift$   & $\nbb$ & time\\
\hline
\multirow{4}{*}{$9922699028$} & $1$   &$\#s = 4$ & $9$ & $47$   &   $241 \, s$ \\
  & $1$   &$\#s = 4, d_{\sqrtop} = 4$ & $3$ & $39$   &   $190 \, s$ \\
  & $1$   &$\#s = 1, d_{\sqrtop} = 4$ & $1$ & $170$   &  $1080 \, s$ \\
   & $1$   &$d_{\arctan} = 4, d_{\sqrtop} = 4$ & $2$ & $14$   &   $244 \, s$ \\
\hline			
\multirow{2}{*}{$3318775219$} & $1$   &$\#s = 2$ & $9$& $338$   &   $1560 \, s$ \\
 & $1$   &$d_{\arctan} = 4, d_{\sqrtop} = 4$ & $2$& $266$   &   $4423 \, s$ \\
\hline	
$7726998381$ & $3$   & $\#s = 4$ & $15$&  $70$  & $2580 \, s$ \\
$7394240696$ & $3$    & $\#s = 2$ & $15$& $351$  & $6480 \, s$\\
$4652969746\_1$ & $6$    &$\#s = 4$ & $15$ & $81$  & $4680 \, s$\\
$\text{OXLZLEZ}\, 6346351218\_2\_0$ &  $6$ & $\#s = 4$& $24$  & $200$  & $20520 \, s$\\
\hline						
\end{tabular}	
\label{table:templates_fly}		
\end{center}	
\end{table}
\begin{table}[!ht]
\small
\vspace*{-0.2cm}
\begin{center}
\caption{Comparison results for random examples using either $\intsolver$ or $\templateoptim$ with $k = 3$ and $\#s = \#s_{\arctan} = 3$}
\begin{tabular}{c|c|rr}
\hline
\multirow{2}{*}{$n$} & \multirow{2}{*}{$m$} & $\templateoptim$ & $\intsolver$\\
& & (time) &  (time)\\
\hline            
 $3$ &  $0.4581 $ & $3.8 \, s$  & $15.5 \, s $ \\
 $4$ & $0.4157 $ & $12.9 \, s $   & $172.1 \, s $ \\
 $5$ & $0.4746 $ & $60 \, s $    & $612 \, s $ \\
 $6$ &  $0.4476 $ & $276 \, s$  & $12240 \, s $ \\
\hline
\end{tabular}
\label{table:rnd}
\end{center}
\end{table}
In Table~\ref{table:templates_fly}, we present some test results for several non-linear Flyspeck inequalities. The integer $\nt$ represents the number of transcendental univariate nodes in the corresponding abstract syntax trees. 
These inequalities are known to be tight and involve sum of arctan of correlated functions in many variables, whence we keep high the number of lifting variables to get precise semialgebraic approximations. However, some inequalities (e.g. $9922699028$) are easier to solve by using coarser semialgebraic approximations.
The first line ($\lift = 9$) corresponds to the algorithm described in~\cite{victorecc}. The second and third line illustrate our improved template method. For the former ($\lift = 3$),  we use no lifting variables to represent square roots of univariate functions. For the latter ($\lift = 1$), we use the $\saapprox$ variant of Sect.~\ref{sec:l1_under} to approximate from below the semialgebraic function $\frac{ \partial_4 \Delta \xb }{\sqrt{4 x_1 \Delta \xb}}$ with the approximation $h_{4 2}$ (see Example~\ref{ex:9922_sa_l1}), so that we save two more lifting variables. Note that this $\saapprox$ variant provides a systematic way to generate polynomial template approximations for semialgebraic functions but it is difficult in practice to generate precise approximations, due to the high computational cost of solving SOS relaxation $(P_{d k})$ for large $d$ and $k$.

In Table~\ref{table:rnd}, we compared our algorithm with the MATLAB toolbox $\texttt{intsolver}$ \cite{Montanher09} (based on the Newton interval method~\cite{Hansen198389}) for random inequalities involving two transcendental functions. 
The functions that we consider are of the form $\xb \mapsto \arctan(p(\xb)) + \arctan(q(\xb))$, where $p$ is a four-degree polynomial and $q$ is a quadratic form. All variables lie in $[0, 1]$. Both $p$ and $q$ have random coefficients (taken in $[0,1]$) and are sparse. The results indicate that for such random examples, our method also outperforms interval arithmetic.
%
\section{Conclusion}
\label{sec:conclusion}
The present nonlinear template method computes certified lower bounds
for global optimization problems. It can provide tight minimax or maxplus
semialgebraic approximations to certify non-linear inequalities involving
transcendental multivariate functions. Our algorithms can solve both small and intermediate size inequalities of the Flyspeck project as well as global optimization problems issued from the literature, with 
a moderate order of SOS relaxation. 

The proposed approach
bears some similarity with the ``cutting planes''  proofs in 
combinatorial optimization, the cutting planes being now
replaced by nonlinear inequalities.  
It also allows one to limit the growth of the number of
lifting variables as well as of polynomial constraints to be handled
in the POP relaxations, at the price of a coarser approximation.
Thus, our method is helpful when the size of optimization problems
increases. Indeed, the coarse lower bounds obtained (even
with a low SOS relaxation order) are better than those obtained with
interval arithmetic or high-degree polynomial approximation. 

Further research would be to apply nonlinear templates to discrete-time optimal control problems, e.g. the Mayer problem. In this case, each set of reachable vectors is abstracted by a template and computed in a forward fashion, by exploiting the dynamics, whereas the templates are refined in a backward fashion. Hence, the template method includes as a special case a set theoretical version of the familiar state/co-state optimality conditions in control.


\appendix

%
\vspace*{-0.3cm}
\section{Appendix: global optimization problems issued from the literature}
\label{sec:appa}
The following test examples are taken from Appendix B in~\cite{Ali:2005:NES:1071327.1071336}. Some of these examples involve functions that depend on numerical constants, the values of which can be found there. 
\begin{itemize}
\item {\em Hartman 3 (H3)}: $\min\limits_{\xb \in [0, 1]^3} f(\xb) = - \sum\limits_{i=1}^4 c_i \exp \left[- \sum\limits_{j=1}^3 a_{i j} (x_j - p_{i j})^2\right]\enspace. $
\item {\em Hartman 6 (H6)}: $\min\limits_{\xb \in [0, 1]^6} f(\xb) = - \sum\limits_{i=1}^4 c_i \exp \left[- \sum\limits_{j=1}^6 a_{i j} (x_j - p_{i j})^2\right] \enspace .$
\item {\em Mc Cormick (MC)}, with $K = [-1.5, 4] \times [-3, 3] $:\\ 
$\min\limits_{\xb \in K} f(\xb) = \sin (x_1 + x_2) + (x_1 - x_2)^2 - 1.5 x_1 + 2.5 x_2 + 1 \enspace.$
\item {\em Modified Langerman (ML)}:\\
$\min\limits_{\xb \in [0, 10]^n} f(\xb) = \sum\limits_{j=1}^5 c_j \cos (d_j / \pi) \exp (- \pi d_j) $, with $d_j = \sum\limits_{i=1}^n (x_i - a_{j i})^2 \enspace .$
\item {\em Schwefel Problem (SWF)}: $\min\limits_{\xb \in [1, 500]^n} f(\xb) = - \sum_{i = 1}^{n} x_i \sin (\sqrt{x_i}) \enspace .$
\end{itemize}
\vspace*{-0.3cm}
\section{Appendix: proofs}
\subsection{Preliminary Results}
\label{pr:prelim}

For the sequel, we need to recall the following definition.
\begin{definition}[Modulus of continuity]
\label{def:modulus_continuity}
Let $\funi$ be a real univariate function defined on an interval $I$. The modulus of continuity of $\funi$ is defined as:
\[\omega (\delta) := \mathop{\sup_{x_1, x_2 \in I}}_{\mid x_1 - x_2 \mid < \delta} \mid \funi(x_1) - \funi(x_2) \mid\]
\end{definition}

We shall also prove that $\uopapprox$ and $\saapprox$ return uniformly convergent approximations nets:
\begin{proposition}
\label{approx_cvg_hyps}
Suppose that Assumption~\ref{minimaxapprox_cvg_hyps} holds. For every function $r$ of the dictionary $\setD$, defined on a closed interval $I$, the procedure $\uopapprox$ returns two nets of univariate lower semialgebraic approximations $(r_p^-)_{p \in \precset}$ and upper semialgebraic approximations $(r_p^+)_{p \in \precset}$, that uniformly converge to $r$ on $I$.

For every semialgebraic function $\fsa \in \setA$, defined on a compact semialgebraic set $K$, the procedure $\saapprox$ returns two nets of lower semialgebraic approximations  $(t_p^-)_{p \in \precset}$ and upper semialgebraic approximations $(t_p^+)_{p \in \precset}$, that uniformly converge to $\fsa$ on $K$.
\end{proposition}

\begin{proof}
First, suppose that the precision $\precision$ is the best uniform polynomial approximation degree.
By Assumption~\ref{minimaxapprox_cvg_hyps}, the procedure $\uopapprox$ returns the sequence of degree-$d$ minimax polynomials, using the algorithm of Remez. This sequence uniformly converges to $r$ on $I$, as a consequence of Jackson's Theorem~\citep[Chap. 3]{Gil:2007:NMS:1355339}.  
Alternatively, when considering maxplus approximations in which the precision is determined
by certain sets of points, we can apply Theorem~\ref{th:maxpluslinf} that implies the uniform convergence of the maxplus approximations.

Next, for sufficiently large relaxation order, the $\saapprox$ procedure returns the best (for the $\lone$ norm) degree-$d$ polynomial under-approximation of a given semialgebraic function, as a consequence of Theorem~\ref{th:cvg_sos_l1}.
\qed
\end{proof}

\subsection{Proof of Lemma~\ref{th:psa}}
\label{pr:psa}
Let us equip the vector space $\R_d[\xb]$ of polynomials $h$ of degree at most $d$ 
with the norm $\linfnorm{h} := \sup_{|\alphab| \leq d} \{ |h_{\alphab}| \}$.

Let $H$ be the admissible set of Problem~$(\Psa)$. Observe that  $H$ is closed in the topology of the latter norm. Moreover, the objective
function of Problem~$(\Psa)$ can be written as $\phi : h\in H  \mapsto \lonenormK{\fsa - h}$, where $\lonenormK{\cdot}$ is the norm of the space $L^1(K,\lambda_n)$.
The function $\phi$ is continuous in the topology of $\linfnorm{\cdot}$
(for polynomials of bounded degree, the convergence
of the coefficients implies the uniform convergence
on every bounded set for the associated polynomial functions, and a fortiori the convergence of these polynomial functions
in $L^1(K,\lambda_n)$). Note  also that $\int_{[0, 1]^n} h\  d \lambda_n= \int_{[0, 1]^n} h(\xb) \ d \lambda_n(\xb) = \int_{[0, 1]^{n + p}} h(\xb,\zb) \ d \lambda_{n + p}(\xb,\zb)$.
We claim that for every $t \in \R$, the sub-level set
$S_t:=\{ h \in H \mid \phi(h) \leq t\}$ is bounded.
Indeed, when $\phi(h) \leq t$, we have:
\[ \lonenormK{h} \leq  \lonenormK{\fsa - h} + \lonenormK{\fsa} \leq t +\lonenormK{\fsa} \enspace .\]
Since on a finite dimensional vector space, all the norms are equivalent,
there exists a constant $C>0$ such that $\linfnorm{h} \leq C \lonenormK{h} $
for all $h\in H$, so 
we deduce that $\linfnorm{h}\leq C(t + \lonenormK{\fsa})$ for all $h\in S_t$, 
which shows the claim. Since $\phi$ is continuous, it follows that every sublevel set of $\phi$, which is a closed bounded subset of a finite dimensional
vector space, is compact. Hence, the minimum of Problem~$(\Psa)$ is attained.
\qed

\subsection{Proof of Proposition~\ref{th:approx_cvg}}
\label{pr:approx_cvg}
The proof is by induction on the structure of $t$.
\begin{itemize}

\item When $t$ represents a semialgebraic function of $\setA$, the under-approximation (\resp{}over-approximation) net $(t_{p}^-)_{p }$ (\resp{}$(t_{p}^+)_{p}$) converges uniformly to $t$ by Proposition~\ref{approx_cvg_hyps}.

\item The second case occurs when the root of $t$ is an univariate function $r \in \setD$ with the single child $c$. Suppose that $r$ is increasing without loss of generality. We consider the net of under-approximations $(c_p^-)_{p}$ (\resp{}over-approximations $(c_p^+)_{p}$) as well as lower and upper bounds $m_{c_p}$ and $M_{c_p}$ which are obtained recursively. Since $K$ is a compact semialgebraic set, one can always find an interval $I_0$ enclosing the values of $r_p^+$ (\ie{}such that $[m_{c_p}, M_{c_p}] \subset I_0$), for all $p$.

The induction hypothesis is the uniform convergence of $(c_p^-)_{p}$ (resp. $(c_p^+)_{p}$) to $c$ on $K$. 
Now, we prove the uniform convergence of $(t_p^+)_{p}$ to $t$ on $K$. One has:

\begin{equation}
\label{preq:0}
\linfnorm{t - t_p^+}  \leq \linfnorm{r \comp c - r_p^+ \comp c} + \linfnorm{r_p^+ \comp c - t_p^+}.
\end{equation}

Let note $\omega$ the modulus of continuity of $r_p^+$ on $I_0$. Thus, the following holds:

\begin{equation}
\label{preq:1}
\linfnorm{r_p^+ \comp c - r_p^+ \comp c_p^+} \leq \omega (\linfnorm{c - c_p^+}).
\end{equation}

Let $\epsilon > 0$ be given. The univariate function $r_p^+$ is uniformly continuous on $I_0$, thus there exists $\delta > 0$ such that $\omega (\delta) \leq \epsilon / 2$.
\if{
\begin{equation}
\label{preq:2}
\omega (\delta) \leq \epsilon / 2.
\end{equation}
}\fi
Let choose such a $\delta$. By induction hypothesis, there exists a precision $p_0$ such that for all $p \geq p_0$, $\linfnorm{c - c_p^+} \leq \delta$. Hence, using~\eqref{preq:1}, the following holds:
\begin{equation}
\label{preq:3}
\linfnorm{r_p^+ \comp c - r_p^+ \comp c_p^+} \leq \epsilon / 2.
\end{equation}
Moreover, from the uniform convergence of $(r_p^+)_{p \in \N}$ to $r$ on $K$ (by Proposition~\ref{approx_cvg_hyps}), there exists a precision $p_1$ such that for all $p \geq p_1$:
\begin{equation}
\label{preq:4}
\linfnorm{r \comp c - r_p^+ \comp c} \leq \epsilon / 2.
\end{equation}
Using~\eqref{preq:0} together with~\eqref{preq:3} and~\eqref{preq:4} yield the desired result. The proof of the uniform convergence of the under-approximations is analogous.
\item If the root of $t$ is a binary operation whose arguments are two children $c_1$ and $c_2$, then by induction hypothesis, we obtain semialgebraic approximations $c_{1, p}^-$, $c_{2, p}^-$, $c_{1, p}^+$, $c_{2, p}^+$ that verify:
\begin{align}
\lim_{p \rightarrow \infty} \linfnorm{c_1 - c_{1, p}^-} = 0 & , & \lim_{p \rightarrow \infty} \linfnorm{c_1 - c_{1, p}^+} = 0, \label{preq_bop:0} \\
\lim_{p \rightarrow \infty} \linfnorm{c_2 - c_{2, p}^-} = 0 & , & \lim_{p \rightarrow \infty} \linfnorm{c_2 - c_{2, p}^+} = 0.\label{preq_bop:1}
\end{align}
If $\bop = +$, by using the triangle inequality:
\begin{align*}
\linfnorm{c_1 + c_2 - c_{1, p}^- -c_{2, p}^- } \leq  \linfnorm{c_1 - c_{1, p}^-} + \linfnorm{c_2 - c_{2, p}^-}, \\
\linfnorm{c_1 + c_2 - c_{1, p}^+ -c_{2, p}^+ } \leq  \linfnorm{c_1 - c_{1, p}^+} + \linfnorm{c_2 - c_{2, p}^+}.
\end{align*}
Then, the uniform convergence comes from~\eqref{preq_bop:0} and~\eqref{preq_bop:1}.
The proof for the other cases is analogous. \qed
\end{itemize}

\subsection{Convergence of the $\templateoptim$ Algorithm}
\subsubsection{Preliminaries: $\Gamma$ and Uniform Convergence}
To study the convergence of the minimizers of $t_p^-$, we first introduce some background on the $\Gamma$-convergence (we refer the reader to~\cite{maso1993introduction} for more details) and the lower semicontinuous envelope. The topology of $\Gamma$-Convergence is known to be metrizable hence, we shall consider the $\Gamma$-Convergence of sequences (rather than nets).
\begin{definition}[$\Gamma$-Convergence]
\label{def:gamma_cvg}
The sequence $(t_p)_{p \in \N}$ $\Gamma$-converges to $t$ if the following two conditions hold: 
\begin{enumerate}
\item (Asymptotic common lower bound) For all $\xb \in K$ and all $(\xb_p)_{p \in \N}$ such that $\lim_{p \rightarrow \infty} \xb_p = \xb$, one has $t(\xb) \leq \liminf_{p \rightarrow \infty} t_p (\xb_p)$. 
\item (Existence of recovery sequences) For all $\xb \in K$, there exists some $(\xb_p)_{p \in \N}$ such that $\lim_{p \rightarrow \infty} \xb_p = \xb$ and $\limsup_{p \rightarrow \infty} t_p (\xb_p) \geq  t(\xb)$. 
\end{enumerate}
\end{definition}
Define $\Rbar := \R \cup \{-\infty, \infty\}$ to be the extended real number line.
\begin{definition}[Lower Semicontinuous Envelope]
\label{def:lsc_env}
Given $t : K \mapsto \Rbar$, the lower semicontinuous envelope of $t$ is defined by:
\[ \tlsc (\xb) := \sup \{  g(\xb) \mid g : K \mapsto \Rbar \text{ is lower semicontinuous and } g \leq f \text{ on } K  \}. \]
\end{definition}

If $t$ is continuous, then $\tlsc := t$.

\begin{theorem}[Fundamental Theorem of $\Gamma$-Convergence~\cite{maso1993introduction}]
\label{th:gamma_cvg}
Suppose that the sequence $(t_p)_{p \in \N}$ $\Gamma$-converges to $t$ and $\xb_p$ minimizes $t_p$. Then every limit point of the sequence  $(\xb_p)_{p \in \N}$ is a global minimizer of $t$.
\end{theorem}

\begin{theorem}[$\Gamma$ and Uniform Convergence~\cite{maso1993introduction}]
\label{th:gamma_prop}
If $(t_p)_{p \in \N}$ uniformly converges to $t$, then $(t_p)_{p \in \N}$ $\Gamma$-converges to $\tlsc$.
\end{theorem}\index{$\Gamma$-convergence}

Theorem~\ref{th:gamma_prop} also holds for nets, since the topology of $\Gamma$-Convergence is metrizable.\index{metrizable}

\subsubsection{Proof of Corollary~\ref{th:approx_cvg_minimizer}}
\label{pr:approx_cvg_minimizer}
From Proposition~\ref{th:approx_cvg}, the under-approximations net $(t_p^-)_{p \in \N}$ uniformly converge to $t$ on $K$. Then, by using Theorem~\ref{th:gamma_prop}, the net $(t_p^-)_{p \in \N}$ $\Gamma$-converges to $\tlsc := t$ (by continuity of $t$). It follows from the fundamental Theorem of $\Gamma$-Convergence~\ref{th:gamma_cvg} that every limit point of the net of minimizers $(\xb_p^*)_{p \in \N}$ is a global minimizer of $t$ over $K$.
\qed
\if{
\begin{figure}
 \includegraphics{example.eps}
\caption{Please write your figure caption here}
\label{fig:1}       
\end{figure}
%
\begin{figure*}
  \includegraphics[width=0.75\textwidth]{example.eps}
\caption{Please write your figure caption here}
\label{fig:2}       
\end{figure*}
%
\begin{table}
\caption{Please write your table caption here}
\label{tab:1}       
\begin{tabular}{lll}
\hline\noalign{\footnotesizeskip}
first & second & third  \\
\noalign{\footnotesizeskip}\hline\noalign{\footnotesizeskip}
number & number & number \\
number & number & number \\
\noalign{\footnotesizeskip}\hline
\end{tabular}
\end{table}
}\fi


\end{document}